\documentclass[11pt]{amsart}
\usepackage[letterpaper,margin=1in]{geometry}                
\usepackage{amssymb,amsmath,amsfonts}
\usepackage{xcolor}
\usepackage[colorlinks,
             linkcolor=blue,
             citecolor=green!70!black,
             pdfproducer={pdfLaTeX},
             pdfpagemode=None,
             bookmarksopen=true
             bookmarksnumbered=true]{hyperref}
\usepackage{graphicx}

\usepackage{tikz}
\usetikzlibrary{decorations.markings,arrows,decorations.pathreplacing}

\usepackage{arydshln}

\allowdisplaybreaks

\newcommand\arXiv[1]{\href{http://arxiv.org/abs/#1}{\nolinkurl{arXiv:#1}}}
\newcommand\MRnumber[1]{\href{http://www.ams.org/mathscinet-getitem?mr=#1}{\nolinkurl{MR#1}}}
\newcommand\DOI[1]{\href{http://dx.doi.org/#1}{\nolinkurl{DOI:#1}}}
\newcommand\MAILTO[1]{\href{mailto:#1}{\nolinkurl{#1}}}

\newtheorem{dummy}{dummy}[section]
\newtheorem{theorem}[dummy]{Theorem}

\newtheorem{lemma}[dummy]{Lemma}

\usepackage{dsfont}
\renewcommand\mathbb\mathbf

\newcommand\bC{\mathbb C}

\newcommand\bF{\mathbb F}

\newcommand\bR{\mathbb R}

\newcommand\bZ{\mathbb Z}

\newcommand\rA{\mathrm A}
\newcommand\rB{\mathrm B}
\newcommand\rC{\mathrm C}
\newcommand\rD{\mathrm D}

\newcommand\rH{\mathrm H}

\newcommand\rO{\mathrm O}

\newcommand\rU{\mathrm U}

\DeclareMathOperator\homology{H}
\renewcommand\H{\homology}

\numberwithin{equation}{section}

\newcommand\longto\longrightarrow
\newcommand\mono\hookrightarrow
\newcommand\epi\twoheadrightarrow

\newcommand\<\langle
\renewcommand\>\rangle
\newcommand\sminus\smallsetminus

\newcommand\Co{\mathrm{Co}}
\newcommand\Leech{\mathrm{Leech}}

\newcommand\Perm{\mathrm{Perm}}

\newcommand\Spin{\mathrm{Spin}}
\newcommand\String{\mathrm{String}}

\newcommand\SU{\mathrm{SU}}
\newcommand\SO{\mathrm{SO}}

\DeclareMathOperator\Aut{Aut}
\DeclareMathOperator\Sym{Sym}
\DeclareMathOperator\Alt{Alt}

\newcommand\golay{\mathfrak{g}_{24}}

\DeclareMathOperator\Sq{Sq}

\newcommand\define[1]{\emph{#1}}

\newcommand{\GL}{\mathrm{GL}}
\newcommand{\SL}{\mathrm{SL}}

\DeclareMathOperator{\trace}{trace}
\newcommand\tr\trace
\newcommand{\VL}{V^L}

\title{\texorpdfstring{$\H^4(\Co_0;\bZ) = \bZ/{24}$}{H^4(Co_0;Z) = Z/24}}
\author{Theo Johnson-Freyd}
\address{Perimeter Institute for Theoretical Physics, Waterloo, Ontario}
\email{\MAILTO{theojf@pitp.ca}}
\author{David Treumann}
\address{Department of Mathematics, Boston College, Boston, Massachusetts}
\email{\MAILTO{treumann@bc.edu}}
\date{\today}

\begin{document}
\begin{abstract}
We show that the fourth integral cohomology of Conway's group $\Co_0$ is a cyclic group of order $24$, generated by the first fractional Pontryagin class of the $24$-dimensional representation.
\end{abstract}
\maketitle

Let $\Co_0 = 2.\Co_1$ denote the linear isometry group of the Leech lattice, the largest of the Conway groups.  By definition, it has a 24-dimensional complex representation which we will denote by $\Leech \otimes \bC$.  The main result of this paper is the following:
\begin{theorem}
\label{conway theorem}
The group cohomology $\H^4(\Co_0;\bZ)$ is isomorphic to $\bZ/24$.  Furthermore
\begin{enumerate}
\item The Chern class $c_2(\Leech \otimes \bC) \in \H^4(\Co_0;\bZ)$ generates a subgroup of index $2$.  
\item There exists a subgroup $\mathrm{CSD} \subset \Co_0$ of order $48$, for which the restriction map $\H^4(\Co_0;\bZ) \to \H^4(\mathrm{CSD};\bZ)$ is injective.  $\mathrm{CSD}$ is isomorphic to $\bZ/3 \times 2D_8$, the product of the cyclic group of order $3$ and the ``binary dihedral'' or ``generalized quaternion'' group of order 16.
\end{enumerate}
\end{theorem}
\noindent
As Theorem~\ref{conway corollary}, we conclude that $\H^4(\Co_1;\bZ)$ is isomorphic to $\bZ/12$.  In Section \ref{sec:second-chern-classes}, we compute the Chern classes $c_2(V)$ for each complex irreducible representation $V$ of $\Co_0$ --- they all have the form $c_2(V) = k(V) c_2(\Leech \otimes \bC)$ for $k(V) \in \bZ/12$.

We note that the fact that the $p$-primary part of $\H^4(\Co_0;\bZ)$ vanishes for $p > 3$ is not new --- the entire cohomology ring $\H^\bullet(\Co_0;\bZ[\frac{1}{2},\frac{1}{3}]) \cong \H^\bullet(\Co_1;\bZ[\frac{1}{2},\frac{1}{3}]) $, and the subring generated by Chern classes, was obtained by the late C. Thomas in \cite[\S 3]{MR2681787}.
\medskip

These days, similar results to Theorem \ref{conway theorem} are sometimes obtained by software advances --- for example \cite{SE09} determined 
\begin{equation}
\label{eq:Ellis-M24}
\H^4(M_{24};\bZ) = \bZ/12 \qquad \text{(where $M_{24}$ denotes the largest Mathieu group)}
\end{equation}
by developing HAP.  But the Conway group is large, and we believe that there is no known chain model for $\H^4(\Co_0;\bZ)$ that can be feasibly handled by a computer.  

Our argument actually appeals directly to \eqref{eq:Ellis-M24}, as well as to some GAP-assisted  (more specifically, Derek Holt's ``Cohomolo'' package) computations of $\H^1$ and $\H^2$ of Mathieu groups with twisted coefficients.  And we have used GAP and Sage extensively while exploring the Conway group.  But formally, our proof is less direct.  It is largely based on analyzing a pair of subgroups of $\Co_0$ that contain the $2$- and the $3$-Sylow subgroups, and that split as semidirect products
\[
2^{12}:M_{24},  \qquad 3^6:2M_{12}.
\]
The colon and the use of $2^n$, $3^m$ for commutative groups follows the ATLAS notation, see \S\ref{subsec:ATLAS-notation}.  These two subgroups are closely related to the Niemeier lattices $A_1^{24}$ and $A_2^{12}$, Kneser $2$- and $3$-neighbors of the Leech lattice.  

Our argument is also based on some good luck, as the existence of such a small subgroup $\mathrm{CSD}$ that detects $\H^4$ is not a priori clear, but it's crucial for our computation.  In fact we encounter the same good luck when studying the Mathieu groups $M_{23}$ and $M_{24}$, leading to a less computer-intensive proof of~\eqref{eq:Ellis-M24}; see Section~\ref{section: other groups}. Theorem~\ref{mathieu moonshine anomaly} gives yet another connection between our calculation of $\H^4(\Co_0;\bZ)$ and the group $\H^4(M_{24};\bZ)$ of~\eqref{eq:Ellis-M24}.
\medskip

For any particular finite group $G$, the determination of the group cohomology $\H^\bullet(G;\bZ)$ is a challenging problem in algebraic topology.  The low-degree groups are more accessible, and have concrete group- and representation-theoretical significance with 19th century pedigrees.  For example $\H^2(G;\bZ)$ is the group of one-dimensional characters of $G$, and $\H^3(G;\bZ)$ classifies the twisted group algebras for $G$.  The Pontryagin dual of $\H^3(G;\bZ)$, which the universal coefficient theorem identifies with $\H_2(G)$, is the Schur multiplier of~$G$.

More recently, a similar role has emerged for $\H^4(G;\bZ)$ --- for instance, it classifies monoidal structures on the category of vector bundles on $G$ that have the form of convolution \cite[App.~E]{MooreSeiberg}.  Nora Ganter has proposed to call the Pontryagin dual $\H_3(G)$ the ``categorical'' Schur multiplier of~$G$.  We have named the subgroup $\mathrm{CSD}$ for ``categorical Schur detector.''
\medskip

The notion of spin structure (and of string obstruction) for a representation reveals a little more structure in Theorem \ref{conway theorem} --- a distinguished generator for $\H^4(\Co_0;\bZ) \cong \bZ/24$, called the ``{first fractional Pontryagin class}'' of the defining representation, denoted $\frac{p_1}2(\Leech \otimes \bR)$. Its construction is briefly reviewed in \S\ref{p12 discussion}. In terms of Chern classes, $2 \frac{p_1}{2}(\Leech \otimes \bR) = -c_2(\Leech \otimes \bC)$.

We use the explicit generator to compute the restriction map $\H^4(\Co_0;\bZ) \to \H^4(G;\bZ)$ for some subgroups $G \subset \Co_0$, i.e.\ we compute the fractional Pontryagin class of the $G$-action on $\bR^{24}$. In particular in Theorem~\ref{restrictions theorem} we compute the restriction to all cyclic subgroups of $\Co_0$, where we find a peculiar relationship with Frame shapes of elements --- this result generalizes the relationship discovered in \cite{GPRV} between cycle types of permutations in $M_{24}$ and $\H^4(M_{24};\bZ)$. In Theorems~\ref{mathieu moonshine anomaly} and~\ref{thm umbral consistency check} we study the restrictions of $\frac{p_1}2(\Leech \otimes \bR)$ to certain ``umbral'' subgroups of $\Co_0$, and relate the answers to the calculations of~\cite{CLW}. A general  connection between $\H^4(\Co_0;\bZ)$ and various forms of moonshine is discussed in~\cite{MR3539377}.
\medskip

We begin with some preliminary remarks in Section~\ref{preliminaries}, in particular recalling the standard transfer-restriction argument that allows theorems like Theorem~\ref{conway theorem} to be proved prime-by-prime. In Section~\ref{large prime} we quickly dispense with the large primes $p\geq 5$. We handle the prime $p=3$ in Section~\ref{prime 3}. The most interesting story occurs at the prime $p=2$, which is the subject of Section~\ref{prime 2}; that section completes the proof of Theorem~\ref{conway theorem}.  Section~\ref{section: other groups} summarizes our proof of Theorem~\ref{conway theorem} and also computes $\H^4(M_{24};\bZ)$, $\H^4(M_{23};\bZ)$, and $\H^4(\Co_1;\bZ)$. 
Section~\ref{sec:second-chern-classes} explains the computation of the map $c_2:R(\Co_0) \to \H^4(\Co_0;\bZ)$ --- the output of the computation is summarized in a table on the last page of that section. 
In Section~\ref{section restrictions} we explain the computation of the restriction of $\frac{p_1}2$ to all cyclic subgroups of $\Co_0$ and Section~\ref{umbral section} discusses the restriction to umbral subgroups.

\subsection*{Acknowledgements}
We thank John Duncan and Miranda Cheng for explaining various aspects of umbral moonshine, and
we are grateful to the Institute for Advanced Study and the Perimeter Institute where parts of this paper were written.  DT's stay at the IAS was supported by a von Neumann fellowship, a Sloan fellowship, and a Boston College faculty fellowship.  Additional support was provided by NSF-DMS-1510444. Research at the Perimeter Institute is supported by the Government of Canada through the Department of Innovation, Science and Economic Development Canada and by the Province of Ontario through the Ministry of Research, Innovation and Science.

\section{Preliminary remarks} \label{preliminaries}

\subsection{Notation}
\label{subsec:ATLAS-notation}
We generally follow the ATLAS \cite{ATLAS} for notation for finite groups, and regularly refer to it (often without mention) for known facts. The cyclic group of order $n$ is denoted variously $\bZ/n$, $\bF_n$ (when $n$ is prime and we are thinking of it as a finite field), and just~``$n$.'' Elementary abelian groups are denoted $n^k$ and extraspecial groups $n^{1+k}$. An extension with normal subgroup $N$ and cokernel $G$ is denoted $N.G$ or occasionally $NG$; an extension which is known to split is written $N:G$. The conjugacy classes of elements of order $n$ in a group $G$ are named $n\rA$, $n\rB$, \dots, ordered by increasing size of the class (decreasing size of the centralizer).

When $G$ is a finite group and $A$ is a $G$ module we write $\H^{\bullet}(G;A)$ for the group cohomology of $G$ with coefficients in $A$.  But when $G$ is a Lie group we will write $\H^{\bullet}(BG)$ to avoid confusion with the cohomology of the manifold underlying $G$.

\subsection{The Conway group} \label{preliminaries on Conway group}
We now recite some standard material about $\Co_0$. By definition, Conway's largest group $\Co_0 = 2.\Co_1$ is the automorphism group of the Leech lattice.  Its order factors as
\begin{equation}
\label{eq:conway-order}
|\Co_0| = 2^{22} \cdot 3^9 \cdot 5^4 \cdot 7^2 \cdot 11 \cdot 13 \cdot 23.
\end{equation}
The Leech lattice is the unique rank-$24$ self-dual even lattice with no roots. It can be constructed in many ways \cite[Ch. 24]{MR1662447}. One standard construction of the Leech lattice begins with the Golay code $\golay : 2^{12} \mono 2^{24}$ \cite{Golay}, which is the unique Lagrangian subspace for the standard (``Euclidean'') inner product on $2^{24} = \bF_2^{24}$ containing no vectors of Hamming length $4$.  
 The subgroup of $S_{24}$ preserving the Golay code is the largest Mathieu group~$M_{24}$. 
Let $\operatorname{Nie}(A_1^{24}) \subset \bR^{24}$ denote the Niemeier lattice with root system $A_1^{24} = (\sqrt 2 \bZ)^{24}$. It is constructed from $\golay$ as the pullback
$$ \begin{tikzpicture}[anchor=base]
  \path (0,0) node (N) {$\operatorname{Nie}(A_1^{24})$} (3,0) node (E) {$\bigl(\frac{\sqrt 2}2\bZ\bigr)^{24}$}
    (0,-2) node (G) {$2^{12}$} (3,-2) node (B) {$2^{24}$};
  \draw [right hook->] (G) -- node[auto] {$\scriptstyle \golay$} (B);
  \draw [->>] (E) -- node[auto] {$\scriptstyle \mod (\sqrt 2 \bZ)^{24}$} (B);
  \draw [->] (N) -- (N -| E.west); \draw [->] (N) -- (G);
  \path (.5,-.5) node {$\ulcorner$};
\end{tikzpicture} $$
By construction, the automorphism group of $\operatorname{Nie}(A_1^{24})$ contains (and in fact is equal to) the semidirect product $2^{24} : M_{24}$, where $M_{24}$ acts by permuting the coordinates and $2^{24}$ acts by basic reflections. The subgroup $2^{12} : M_{24}$ (in which $2^{12} \subset 2^{24}$ via the Golay code) preserves a unique index-2 sublattice $L$ of $\operatorname{Nie}(A_1^{24})$. This sublattice $L$ can be extended to a self-dual lattice in exactly three ways: it can be extended back to $\operatorname{Nie}(A_1^{24})$; it can be extended to an odd lattice (the so-called ``odd Leech lattice'' discovered by \cite{MR0010153}); and it can be extended to a new even lattice. The third of these is by definition the \define{Leech lattice}.

By construction, then, $\Co_0 = \Aut(\text{Leech lattice})$ contains a subgroup of shape $2^{12}:M_{24}$.  As the order of $M_{24}$ is divisible by $2^{10}$, a $2$-Sylow subgroup of $2^{12}:M_{24}$ has order $2^{22}$ and is also a $2$-Sylow subgroup of $\Co_0$.  A similar construction of the Leech lattice using the ternary Golay code $3^6 \mono 3^{12}$ and the Niemeier lattice with root system $A_2^{12}$ provides $\Co_0$ with a subgroup of shape $3^6:2M_{12}$ containing the $3$-Sylow. (It extends to a maximal subgroup of shape $2 \times (3^6:2M_{12})$.) The complete list of maximal subgroups of $\Co_0$ was worked out in \cite{MR723071}.

\subsection{Transfer-restriction}

These large subgroups that contain Sylows are very useful for computing cohomology of finite groups, because of the following standard result.

\begin{lemma}\label{transfer restriction}
 Let $G$ be a finite group. Then $\H^k(G;\bZ)$ is finite abelian for $k\geq 1$, and so splits as $\H^k(G;\bZ) = \bigoplus_p \H^k(G;\bZ)_{(p)}$ where the sum ranges over primes $p$ and $\H^k(G;\bZ)_{(p)}$ has order a power of $p$.  Fix a prime $p$ and suppose that $S \subset G$ is a subgroup such that $p$ does not divide the index $|G|/|S|$, i.e.\ such that $S$ contains the $p$-Sylow of $G$. Then the \define{restriction} map $\alpha \mapsto \alpha|_S \H^k(G;\bZ)_{(p)} \to \H^k(S;\bZ)_{(p)}$ is an injection onto a direct summand.
\end{lemma}

\begin{proof}
Define the \define{transfer} map $\alpha \mapsto \mathrm{tr}_{G/S} \alpha : \H^k(S;\bZ) \to \H^k(G;\bZ)$ by summing over the fibers of the finite covering $\rB S \to \rB G$ \cite[\S XII.8]{MR0077480}. The composition $\alpha \mapsto \mathrm{tr}_{G/S}(\alpha|_S)$ is multiplication by $|G|/|S|$, and so is invertible on $\H^k(G;\bZ)_{(p)}$.
\end{proof}

Thus, in order to understand $\H^4(\Co_0;\bZ)$, we may work prime by prime. The only primes that participate are those that divide $|\Co_0|$ \eqref{eq:conway-order}. It is known \cite{MR2681787} that $\H^4(\Co_0;\bZ)_{(p)} = 0$ for $p\geq 5$ (we will also verify this directly).  As already mentioned, subgroups containing the 2- and 3-Sylows are $2^{12}:M_{24}$ and $3^6:2M_{12}$.

\subsection{Fractional Pontryagin class} \label{p12 discussion}

Let $B\SO$ and $B\Spin$ denote the homotopy colimits $\varinjlim_n B\SO(n)$ and $\varinjlim_n B\Spin(n)$, respectively. Then $\H^4(B\SO;\bZ)$ and $\H^4(B\Spin;\bZ)$ are both isomorphic to $\bZ$. The former group is generated by the first Pontryagin class $p_1$. The restriction map $\H^4(B\SO;\bZ) \to \H^4(B\Spin;\bZ)$ is multiplication by $2$, and so the generator of $\H^4(B\Spin;\bZ)$ is called the first \define{fractional Pontryagin class} and denoted $\frac{p_1}2$.
 The restriction maps $\H^4(B\SO;\bZ) \to \H^4(B\SO(n);\bZ)$ and $\H^4(B\Spin;\bZ) \to \H^4(B\Spin(n);\bZ)$ are isomorphisms for $n\geq 5$.
 
Suppose that $G$ is a finite group and $V : G \to \Spin(n)$ is a spin representation. The fractional Pontryagin class of $V$, denoted $\frac{p_1}2(V) \in \H^4(G;\bZ)$, is the pullback of $\frac{p_1}2$ along $V$. This class is also called the ``String obstruction'' because of its relation to the question of lifting a homomorphism $V : G \to \Spin(n)$ to a loop space map $G \to \String(n)$, where, for $n\geq 5$, $\String(n)$ is the 3-connected cover of $\Spin(n)$.  
(According to a first-hand account by Chris Douglas, the name ``$\String(n)$'' for this topological group is due to Thomas Goodwillie. For discussion of $\String(n)$, see \cite{MR2800361}.)
This is
 analogous to the role that the second Stiefel--Whitney class $w_2(V) \in \H^2(G;\bZ/2)$ plays in measuring whether an oriented representation $V : G \to \SO(n)$ lifts to $\Spin(n)$.

If $V : G \to \rO(n)$ is merely a real representation, then $\frac{p_1}2(V)$ need not be defined: it is easy to come up with examples where $p_1(V)$ is odd.
Suppose that $V$ admits a lift to $\Spin(n)$, but that such a lift has not been chosen.
 The recipe for defining $\frac{p_1}2(V)$ above makes it seem that its value might depend on the choice of lift. In fact, $\frac{p_1}2(V)$ is well-defined for real representations admitting a lift to $\Spin(n)$ --- it does not depend on the choice of spin structure. Moreover, if $V_1$ and $V_2$ are two spin representations, then $\frac{p_1}2(V_1 \oplus V_2) = \frac{p_1}2(V_1) + \frac{p_1}2(V_2)$. One can prove these claims by studying the problem of lifting directly from $\rO(n)$ to $\String(n)$ and showing that the obstruction lives in a certain generalized cohomology theory named ``supercohomology''  by~\cite{GuWen,WangGu2017}.

Since $\Co_0$ is the Schur cover of a simple group, both $\H^2(\Co_0;\bZ)$ and $\H^3(\Co_0;\bZ)$ vanish, and so every real representation $V:\Co_0 \to \mathrm{O}(n)$ lifts uniquely to a spin representation $V:\Co_0 \to \Spin(n)$.

\section{The large primes \texorpdfstring{$p \geq 5$}{p>5}} \label{large prime}

We now check that $\H^4(\Co_0;\bZ)_{(p)} = 0$ for $p \geq 5$, confirming the calculation from \cite{MR2681787}.  It is equivalent to show that $\H^4(\Co_1;\bZ)_{(p)} = 0$ for $p \geq 5$, since the pullback map $\H^\bullet(\Co_1) \to \H^\bullet(\Co_0)$ is an isomorphism on odd parts.

\begin{lemma} \label{lemma p7}
  $\H^4(\Co_0;\bZ)_{(p)} = 0$ for $p\geq 7$.
\end{lemma}

\begin{proof}
There is one conjugacy class each in $\Co_1$ of order $11$ and $13$, and two of order $23$; the $p$-Sylow subgroups for $p = 11$, $13$, and $23$ are cyclic. It follows (Pigeonhole) that, for $g\in \Co_1$ of order $p$, there exists $a \neq \pm 1\in \bZ/p^\times$ such that $g$ is conjugate to $g^a$. The automorphism $g \mapsto g^a$ acts on $\H^4(\langle g \rangle;\bZ) \cong \bZ/p$ by multiplication by $a^2 \neq 1$, and so has no fixed points. By Lemma \ref{transfer restriction}, $\H^4(\Co_1;\bZ)_{(p)}$ injects into the conjugation-in-$\Co_1$-fixed subgroup of $\H^4(\langle g \rangle;\bZ)$, which is trivial.

A similar argument handles the prime $p=7$. Indeed, the $7$-Sylow in $\Co_1$ is a copy of $(\bZ/7)^2$ and is contained in a subgroup isomorphic to $L_2(7)^2$ (following the ATLAS \cite{ATLAS}, $L_2(7)$ denotes the simple group $\mathrm{PSL}_2(\bF_7)$). This is in turn contained in a maximal subgroup of shape $(L_2(7) \times A_7):2$. But $L_2(7)$ has a unique conjugacy class of order $7$, and so, just as above, $\H^\bullet(L_2(7);\bZ)_{(7)}$ vanishes in degrees $\bullet \leq 4$. An application of K\"unneth's formula shows that $\H^4(L_2(7)^2;\bZ)_{(7)}$ vanishes, but $\H^4(\Co_1;\bZ)_{(7)} \to \H^4(L_2(7)^2;\bZ)_{(7)}$ is an injection, since $L_2(7)^2$ contains the $7$-Sylow. \end{proof}

For the prime $5$, we need slightly stronger technology. Suppose $G = N.J$ is an extension of finite groups. The \define{Lyndon--Hochschild--Serre (LHS) spectral sequence} is a spectral sequence converging to $\H^\bullet(G;\bZ)$ with $E_2$-page $\H^\bullet(J;\H^\bullet(N;\bZ))$. This $E_2$ page gives an upper bound on the cohomology of $G$, and  often this upper bound suffices.

\begin{lemma}
  $\H^4(\Co_0;\bZ)_{(5)} = 0$.
\end{lemma}

\begin{proof}
$\Co_1$ has a maximal subgroup of shape $G = 5^3:(4\times A_5).2$, which contains the $5$-Sylow. We work out the LHS spectral sequence for this $G$.
The center $\bZ/4$ of $4\times A_5$ acts with nontrivial central character on $\H^2(5^3;\bZ) = 5^3$, $\H^3(5^3;\bZ) = \Alt^2(5^3)$, and $\H^4(5^3;\bZ) \cong \Sym^2(5^3) \oplus \Alt^3(5^3)$.
It follows that the cohomology groups $\H^i(4\times A_5; \H^j(5^3;\bZ))$ vanish for $j\in\{1,2,3,4\}$,
and so the restriction map $\H^4(5^3:(4\times A_5);\bZ) \to \H^4(4 \times A_5;\bZ)$ is an isomorphism.
 Choosing an element $g\in A_5$ of order $5$, the restriction map $\H^4(4 \times A_5;\bZ)_{(5)} \to \H^4(\langle g \rangle;\bZ) = \bZ/5$ is an isomorphism.

But $\Co_1$ has only three conjugacy classes of order $5$,  distinguished by their centralizers, forcing $g$ and $g^2$ to be conjugate in $\Co_1$, where $g$ is the chosen element of order $5$ in $A_5$. Now we may proceed as in Lemma~\ref{lemma p7}: $g \mapsto g^2$ acts as multiplication by $-1$ on $\H^4(\langle g \rangle;\bZ)$, while the restriction map $\H^4(\Co_1;\bZ)_{(5)} \to \H^4(5^3:(4\times A_5);\bZ)_{(5)} \cong \H^4(\langle g \rangle;\bZ)$ is an injection into the conjugation-invariant classes in the latter. Thus $\H^4(\Co_1;\bZ)_{(5)} = 0$.
\end{proof}

\section{The prime \texorpdfstring{$p=3$}{p=3}} \label{prime 3}

Each of $\Co_0$ and $\Co_1$ has four conjugacy classes of order $3$, which following the ATLAS we call $3\rA,3\rB,3\rC,$ and $3\rD$.  They are distinguished by their traces on the Leech lattice:
\[
\trace(3\rA,\Leech) = -12, \quad \trace(3\rB,\Leech) = 6, \quad \trace(3\rC,\Leech) = -3, \quad \trace(3\rD,\Leech) = 0
\]
In this section we will show that the restriction map
\begin{equation}\label{restrict to 3D}
\H^4(\Co_0;\bZ)_{(3)} \to \H^4(\langle3\rD\rangle;\bZ) \cong \bZ/3
\end{equation}
is an isomorphism, where $\langle 3\rD\rangle$ denotes any cyclic subgroup of $\Co_0$  generated by an element in $3\rD$.  
A generator for $\H^4(\langle3\rD\rangle;\bZ) \cong \bZ/3$ is given by $c_2(L \oplus \bar{L})$, where $L$ and $\bar{L}$ are the two nontrivial one-dimensional representations of $\langle 3\rD\rangle$.

\begin{lemma} \label{3 part is detected by class 3D}
The map \eqref{restrict to 3D} is a surjection.
\end{lemma}

\begin{proof}
Since $\trace(3\rD,\Leech) = 0$, $\Leech\otimes\bC$ splits over $\langle 3\rD\rangle$ as 8 copies of each of the three 1-dimensional irreps of $\bZ/3$. From this one computes 
\begin{eqnarray*}
c_2(\Leech \otimes \bC\vert_{3\rD}) & = & c_2(8 \oplus 8L \oplus 8 \bar{L}) \\
& = & 8 c_2(L \oplus \bar{L}) \\
& = & -c_2(L \oplus \bar{L}) \text{ mod $3$}
\end{eqnarray*}
which is nonzero in $\H^4( \langle3\rD\rangle;\bZ)$.
\end{proof}

It remains to prove that \eqref{restrict to 3D} is an injection. 
By the transfer-restriction Lemma~\ref{transfer restriction}, $\H^4(\Co_0;\bZ)_{(3)}$ injects into $\H^4(3^6:2M_{12};\bZ)_{(3)}$, since $3^6:2M_{12} \subset \Co_0$ contains the $3$-Sylow. We first show in Lemma~\ref{36:2M12} that $\H^4(3^6:2M_{12};\bZ)_{(3)} \cong \bZ/3 \oplus \bZ/3$. Then in Lemma~\ref{p3 finish} we show that the map $\H^4(\Co_0;\bZ)_{(3)} \to \H^4(3^6:2M_{12};\bZ)_{(3)}$ is not a surjection. This completes the proof that (\ref{restrict to 3D}) is an isomorphism.

\begin{lemma} \label{36:2M12}
  $\H^4(3^6:2M_{12};\bZ)_{(3)} \cong \bZ/3 \oplus \bZ/3$.
\end{lemma}

In the proof, we will see that the splitting is pretty canonical --- the homomorphisms $2M_{12} \to 3^6:2M_{12} \to 2M_{12}$ realize $\H^4(2M_{12};\bZ)_{(3)}$ as a summand of $\H^4(3^6:2M_{12};\bZ)_{(3)}$, and one knows
\begin{equation}
\label{eq:H42M12}
\H^4(2M_{12};\bZ)_{(3)} \cong \H^4(M_{12};\bZ) \cong \bZ/3.
\end{equation}
Equation \eqref{eq:H42M12}
is an easy task for HAP \cite{SE09}, though it was previously computed by hand by~\cite{MilgramTezuka} and perhaps by others.

\begin{proof}
$2M_{12}$ has two $6$-dimensional nontrivial modules over $\bF_3$, the ternary Golay code and its dual, the \define{cocode}.  These codes are discussed by Golay himself \cite{Golay}, and as representations of $2M_{12}$ by Coxeter in \cite{Coxeter}.  In the subgroup of $\Co_0$ of shape $3^6:2M_{12}$, the module $3^6$ is naturally the code module, and 
\begin{equation}
\label{eq:cocode}
(3^6)^\vee = \H^2(3^6;\bZ)
\end{equation}
is isomorphic to the cocode.
  We will write ``$E$'' for the cocode as a $2M_{12}$-module. The modules $E$ and $E^\vee$ are interchanged by the outer automorphism of $2M_{12}$, and so the distinction between $E$ and $E^\vee$ is not particularly important.

We study the LHS spectral sequence $\H^\bullet(2M_{12};\H^\bullet(3^6;\bZ)) \Rightarrow \H^\bullet(3^6:2M_{12};\bZ)$. 
As $2M_{12}$-modules, we have
\[
\H^2(3^6;\bZ) \cong E, \quad \H^3(3^6;\bZ) \cong \Alt^2(E)
\]
and $\H^4(3^6;\bZ)$ is an extension
\begin{equation}
\label{eq:Kunneth}
0 \to \Sym^2(E) \to \H^4(3^6;\bZ) \to \Alt^3(E) \to 0
\end{equation}
The central $\bZ/2$ in $2M_{12}$ acts trivially on $\Sym^2(E)$ and by the sign on $\Alt^3(E)$, so \eqref{eq:Kunneth} is split --- $\H^4(3^6:\bZ) \cong \Sym^2(E) \oplus \Alt^3(E)$ as $2M_{12}$-modules.

Moreover, since the central $\bZ/2$ acts by the sign representation on $E$ and $\Alt^3(E)$, all cohomology groups $\H^i(2M_{12};E)$ and $\H^i(2M_{12};\Alt^3(E))$ vanish. These remarks and \eqref{36:2M12} show that the LHS spectral sequence for $\H^\bullet(3^6:2M_{12};\bZ)_{(3)}$ begins
\[
\begin{array}{ccccc}
 \H^0(\Sym^2(E)) & \\
 \H^0(\Alt^2(E)) & \H^1(\Alt^2(E)) &  \\
 0 & 0 & 0 & 0 & 0 \\
 0 & 0 & 0 & 0 & 0 \\
 \bZ & 0 & 0 & 0 & \bZ/3 
\end{array}
\]
To save space, in the table we have recorded just the coefficients of the cohomology group, so that $\H^i(V) = \H^i(2M_{12};V)$.   

In fact $H^0(\Sym^2(E))$ (and $H^0(\Alt^2(E))$ also) vanishes --- as $E$ is irreducible an invariant quadratic (resp. symplectic) form on $E$ must be either zero, or nondegenerate.  But such a form cannot be nondegenerate since $E$ is not self-dual as a $2M_{12}$-representation.  Since the bottom row of the spectral sequence is split off by a homomorphism $2M_{12} \to 3^6:2M_{12}$, it suffices to check that
\[
\H^1(2M_{12}; \Alt^2(E)) = \bZ/3.
\]
We have verified this using Cohomolo.
\end{proof}

\begin{lemma} \label{p3 finish}
 The restriction map $\H^4(\Co_0;\bZ)_{(3)} \to \H^4(3^6:2M_{12};\bZ)_{(3)}$ is not a surjection.
\end{lemma}

\begin{proof}
Consider the permutation representation $\Perm:M_{12} \to \rO(12)$, and pull it back (under the same name) to $3^6:2M_{12}$. We will prove the Lemma by proving that the second Chern class $c_2(\Perm) \in \H^4(3^6:2M_{12};\bZ)_{(3)}$ does not extend to $\H^4(\Co_0;\bZ)_{(3)}$. 

The group $2M_{12}$ has two conjugacy classes of elements of order $3$. One acts on the permutation representation of $M_{12}$ with cycle structure $1^3\,3^3$, and the other acts with cycle structure $3^4$. Let $a \in 2M_{12}$ denote a representative of the first class and $b \in 2M_{12}$ a representative of the second. Then 
\[
c_2(\Perm)|_{\langle a \rangle} = 3 c_2(L \oplus \bar L) = 0, \qquad c_2(\Perm)|_{\langle b\rangle} = 4c_2(L \oplus \bar L) \neq 0,\]
where $L$ and $\bar L$ are the two nontrivial one-dimensional representations of $\bZ/3$.

  Under the inclusion $2M_{12} \to 3^6:2M_{12} \to \Co_0$, the elements~$a$ and~$b$ have traces $\tr(a,\Leech) = 6$ and $\tr(b,\Leech) = 0$.  The group $3^6:2M_{12}$ is small enough to completely handle on the computer --- say in terms of its permutation representation of degree $729$. By simply running through all elements of $3^6$ one finds that there are $162$ elements $x \in 3^6$ such that $xa \in 3^6 : 2M_{12}$ has order $3$ and $\tr(xa,\Leech) = 0$. Choose one such $x$. Then $\tr(xa,\Leech) = \tr(b,\Leech)$, so $xa$ and $b$ are conjugate in $\Co_0$. But $c_2(\Perm)|_{\langle xa \rangle} = c_2(\Perm)|_{\langle a \rangle} = 0$. This shows that $c_2(\Perm) \in \H^4(3^6:2M_{12};\bZ)_{(3)}$ distinguishes elements of $3^6:2M_{12}$ that are conjugate in $\Co_0$, and so cannot extend to a class in $\H^4(\Co_0;\bZ)_{(3)}$.
\end{proof}

\section{The prime \texorpdfstring{$p=2$}{p=2}} \label{prime 2}

To complete the proof of Theorem~\ref{conway theorem}, we show that $\H^4(\Co_0;\bZ)_{(2)} \cong \bZ/8$. 
In Lemma~\ref{Co0 order 2 lower bound} we find the group $\mathrm{CSD}$ from part (2) of the Theorem, and use it to give  a lower bound of $8$ on the order of $\frac{p_1}2(\Leech)$. 
Then, using this bound, we show in Lemma~\ref{Co0 even lemma} that $\H^4(2^{12}:M_{24};\bZ)_{(2)} \cong \bZ/8 \oplus \bZ/4$. Finally, in Lemma~\ref{p2}, we show that the restriction map $\H^4(\Co_0;\bZ)_{(2)} \to \H^4(2^{12}:M_{24};\bZ)_{(2)}$ is not a surjection. Since, by transfer-restriction Lemma~\ref{transfer restriction}, its image is a direct summand which, by Lemma~\ref{Co0 order 2 lower bound}, contains an element of order $8$, its image must be isomorphic to $\bZ/8$, completing the proof. Parts (1) and (2) of Theorem~\ref{conway theorem} follow from the isomorphism $\H^4(\Co_0;\bZ) \cong \bZ/24$ together with our proof of Lemma~\ref{Co0 order 2 lower bound}.

\begin{lemma} \label{Co0 order 2 lower bound}
 The order of $\frac{p_1}2(\Leech\otimes \bR)$ is divisible by $8$.
\end{lemma}

\begin{proof}
As $2 \frac{p_1}{2}(\Leech \otimes \bR) = -c_2(\Leech \otimes \bC)$, it suffices to show $c_2(\Leech \otimes \bC)$ has order divisible by $4$.  We will restrict $\Leech \otimes \bC$ to a subgroup of $\Co_0$ isomorphic to the binary dihedral group $2D_8$ of order $16$, double covering the symmetries of the square.  Of the three two-dimensional irreducible representations of $2D_8$, two are faithful and quaternionic.  The other is the two-dimensional real defining representation of $D_8$.  Let us call the quaternionic ones $M$ and $M'$ --- they are exchanged by an outer automorphism of $2D_8$.  Then $M:2D_8 \hookrightarrow \SU(2)$ makes $2D_8$ into one of the McKay subgroups of $\SU(2)$.  Its McKay graph is
\begin{equation}
\label{eq:McKay}
\begin{aligned}
\begin{tikzpicture}
  \path (0,0) node (SW) {$1$}
        (1,1) node (A) {$M$}
        (0,2) node (NW) {$\text{Triv}$}
        (2,1) node (M) {$2$}
        (3,1) node (B) {$M'$}
        (4,2) node (NE) {$1$}
        (4,0) node (SE) {$1$};
  \draw (A) -- (SW); \draw (A) -- (NW); \draw (A) -- (M);
  \draw (B) -- (SE); \draw (B) -- (NE); \draw (B) -- (M);
\end{tikzpicture}
\end{aligned}
\end{equation}
We have indicated the trivial module and the dimensions of the other irreducible modules, all of which are real and factor through $D_8$.  

For any McKay group $M:G \hookrightarrow \SU(2)$, the group $\H^4(G;\bZ)$ is cyclically generated by $c_2(M)$, with order $|G|$.  Thus we may write $c_2(M') = k c_2(M)$ for some integer $k \in \bZ/16$.  In fact $k = 9$ --- to determine this, we note that $M \oplus M'$ is isomorphic to $\Sym^3(M)$, and (as $c_1(M) = c_1(M') = 0$), $c_2(M \oplus M') = c_2(M) + c_2(M')$.  But $\Sym^3$ of the standard representation of $\SU(2)$ on $\bC^2$ has weights $-3,-1,1,$ and $3$, and the standard representation itself has weights $1$ and $-1$, so we compute by the splitting principle
\[
c_2(\Sym^3(\bC^2)) = \left[
\begin{array}{c}
c_1(-3)c_1(-1) + c_1(-3)c_1(1) + c_1(-3)c_1(3) \\
 + c_1(-1)c_1(1) + c_1(-1)c_1(3) + c_1(1)c_1(3)
 \end{array}\right] = 10 \left[c_1(1)c_1(-1)\right] = 10 c_2(\bC^2)
\]
where we have written $c_1(n) \in H^2(B\mathrm{U}(1);\bZ)$ for the Chern class of representation $\mathrm{U}(1) \to \mathrm{U}(1)$ of degree $n$.

We will momentarily find a copy of $2D_8$ inside $\Co_0$ such that the central element $c\in 2D_8$ is the central element of $\Co_0$. {We may compute the restriction of any representation of $\Co_0$ to such a subgroup, even before proving it exists.  In $2D_8$, every conjugacy class of order $4$ squares to $c$, and every conjugacy class of order $8$ has fourth power equal to $c$.  In $\Co_0$, there is a unique conjugacy class of order $4$ that squares to the central element --- it is the unique conjugacy class projecting to 2B in $\Co_1$.  There is also a unique conjugacy class of order $8$ whose fourth power is the central element --- it is the unique conjugacy class projecting to 4E in $\Co_1$.}  On $\Leech \otimes \bC$, the classes of order $4$ and $8$ have trace $0$, and the class of order $2$ has trace $-24$, so from characters of $2D_{8}$ we compute
\[
\Leech \otimes \bC \vert_{2D_8} = 6M \oplus 6M'
\]
so that $c_2(\Leech \otimes \bC)\vert_{2D_8} = 60 c_2(M)$, and $60$ has order $4$ in $\bZ/16$.

To complete the proof, it remains to construct a subgroup $2D_8 \subset \Co_0$ containing the central element.  We originally found one by reducing it to a finite search inside of $2^{12}:M_{24}$, which we then implemented in Sage.  Here is a simpler way which also shows that the $2D_8 \subset \Co_0$ and $\bZ/3 \subset \Co_0$ detecting the $2$- and $3$-parts of cohomology can be chosen to commute with each other; the group $\mathrm{CSD}$ from the statement of Theorem~\ref{conway theorem} is simply the product $2D_8 \times \bZ/3$ for these commuting subgroups.  From \cite[\S 2.2]{MR723071}, we see that the centralizer of the class $3\mathrm{D}$ in $\Co_1$ is $3 \times A_9$, and the centralizer of its lift in $\Co_0$ is $3 \times 2A_9$, where $2A_9$ denotes the Schur cover of the alternating group.  (This group $3 \times A_9$ is the top of the ``Suzuki chain'' of subgroups of $\Co_1$).  The center of $2A_9$ coincides with the center of $\Co_0$, so it suffices to find an inclusion $2D_8 \subset 2A_9$ preserving the centers.
The natural inclusion $A_6 \to A_9$ lifts to an inclusion $2A_6 \to 2A_9$, and there is unique conjugacy class of subgroups of $A_6$ that are isomorphic to $D_8$.  The preimage in $2A_6$ can be seen (we used GAP) to be $2D_8$.
\end{proof}

Note that, assuming Theorem~\ref{conway theorem} has been proved, our proofs of Lemmas~\ref{3 part is detected by class 3D} and~\ref{Co0 order 2 lower bound} further imply that the restriction map $\H^4(\Co_0;\bZ) \to \H^4(\mathrm{CSD};\bZ) \cong \bZ/{48}$ is an injection, verifying part~(2) of the Theorem. Part~(1) is also an immediate consequence of our proof of Lemma~\ref{Co0 order 2 lower bound}.

Lemma~\ref{Co0 order 2 lower bound} provides a lower bound on on $\H^4(\Co_0;\bZ)_{(2)}$; our next step (Lemma~\ref{Co0 even lemma}) in the proof of Theorem~\ref{conway theorem} will be to give an upper bound. We will rely on some background on the cohomology of elementary abelian $2$-groups, which we now review. Let $E$ be an elementary abelian $2$-group and write $E^\vee = \hom(E,\bF_2) \cong \hom(E,U(1))$ for the $\bF_2$-vector space dual to $E$.  The following is standard:
\begin{lemma}\label{cohomology of an elementary 2-group}
 There is an isomorphism of rings $\H^\bullet(E;\bZ/2) \cong \Sym^\bullet(E^\vee)$.
 For $i \geq 1$, the reduction map $\H^i(E;\bZ) \to \H^i(E;\bZ/2)$ is injective.
 The image of $\H^i(E;\bZ)$ in $\Sym^i(E;\bZ/2)$ is the kernel of the Bockstein, or first Steenrod squaring map, $\Sq^1:\H^i(E;\bZ/2) \to \H^{i+1}(E;\bZ/2)$, which acts as a differential on $\Sym^\bullet(E^\vee)$.  Moreover, the long sequence
\[
0 \to E^\vee \xrightarrow{\Sq^1} \Sym^2(E^\vee) \xrightarrow{\Sq^1} \Sym^3(E^\vee) \xrightarrow{\Sq^1} \cdots
\]
is exact. \qed
\end{lemma}

Let us denote the quotient $\Sym^2(E^\vee)/\Sq^1(E^\vee)$ by $\Alt^2(E^\vee)$ --- it is the quotient of $E^\vee \otimes_{\bF_2} E^\vee$ by the subspace generated by tensors of the form $x \otimes x$, which is the standard definition of the exterior square functor is characteristic $2$ \cite[\S III.7]{MR0354207}.  
More generally, in characteristic $2$ the alternating power $\Alt^k(E^\vee)$ is defined as the quotient of the tensor power $(E^\vee)^{\otimes k}$ by the subspace spanned by all tensors that contain a repeated factor --- all tensors of the form $(\cdots \otimes v \otimes \cdots \otimes v \otimes \cdots)$ for $v \in E^\vee$ --- and can be identified $\GL(E)$-equivariantly with the quotient of $\Sym^k(E^\vee)$ by the subspace spanned by monomials that are not squarefree.  We will also use:

\begin{lemma}\label{Alt injecting into Tens}
 $\Alt^k(E^\vee)$ is isomorphic to a $\GL(E^\vee)$-stable subspace of $(E^\vee)^{\otimes k}$.\qed
\end{lemma}
Indeed in characteristic 2, $\Alt^k(E^\vee)$ can be realized as the image of the norm map for the $S_k$-action on $(E^\vee)^{\otimes k}$:
\[
(E^\vee)^{\otimes k} \to (E^\vee)^{\otimes k} \qquad w_1 \otimes \cdots \otimes w_k \mapsto \sum_{\sigma \in S_k} w_{\sigma(1)} \otimes \cdots \otimes w_{\sigma(k)}
\]

Lemma~\ref{cohomology of an elementary 2-group} shows that there are isomorphisms of $\GL(E)$-modules
\begin{equation}
\label{eq:coh of elem 2gp}
\H^2(E;\bZ) \cong E^\vee \text{ and }\H^3(E;\bZ) \cong \Alt^2(E^\vee).
\end{equation}
A little more work gives us a useful description of $\H^4(E;\bZ)$ as well:
\begin{lemma}
\label{lem:filtration}
There is a $3$-step filtration $F_1 \subset F_2 \subset F_3 = \H^4(E;\bZ)$ by $\GL(E)$-submodules, whose associated graded modules are
\[
F_1 \cong E^\vee \qquad F_2/F_1 \cong \Alt^2(E^\vee) \qquad F_3/F_2 \cong \Alt^3(E^\vee)
\] 
\end{lemma}

See \cite[Prop. 2.2]{MR2320456} for another description of $\H^4(E;\bZ)$.

\begin{proof}
Using Lemma \ref{cohomology of an elementary 2-group}, it suffices to prove that the kernel $K$ of $\Sq^1:\Sym^4(E^\vee) \to \Sym^5(E^\vee)$ has such a filtration.  The symmetric algebra is reduced, so the squaring maps
\[
E^\vee \xrightarrow{\Sq^1} \Sym^2(E^\vee) \xrightarrow{\Sq^2} \Sym^4(E^\vee)
\]
(each of which just sends $f$ to $f^2$) are injective.  
The image of $\Sq^2$ is contained in $K$
 --- this may be seen directly, or as an instance of the Adem relation $\Sq^1 \Sq^2  = \Sq^3$, where $\Sq^3$ vanishes on $\H^2(-;\bF_2)$ by definition.  We have already noted the identification $\Sym^2(E^\vee)/\Sq^1(E^\vee) \cong \Alt^2(E^\vee)$.  To show that $K/\Sq^2(E^\vee)$ is isomorphic to $\Alt^3(E^\vee)$, we may note that $K = \Sq^1(\Sym^3(E^\vee))$ and that the preimage of $\Sq^2(\Sym^2(E^\vee))$ under $\Sq^1:\Sym^3(E^\vee) \twoheadrightarrow K$ is the subspace spanned by nonsquarefree monomials.
\end{proof}

With these remarks in hand, we may now turn to our promised upper bound on $\H^4(\Co_0;\bZ)_{(2)}$. 
As discussed already in \S\ref{preliminaries on Conway group}, the 2-Sylow in $\Co_0$ is contained in a maximal subgroup of shape $2^{12}:M_{24}$, where $2^{12}$ denotes the extended Golay code module. This $M_{24}$-module is not irreducible, and it is not isomorphic to its dual module. Following \cite{Ivanov09}, we will write $C_{12} = 2^{12} = 2.C_{11}$ for the extended Golay code module, and $C_{11}$ for its simple quotient. Its dual, the \define{extended cocode} module, is $C_{12}^\vee = C_{11}^\vee.2$ and $C_{11}^\vee$ is its simple submodule. $C_{12}^\vee$ is the unique $12$-dimensional $M_{24}$-module over $\bF_2$ with no fixed points.   

\begin{lemma} \label{Co0 even lemma}
$\H^4(C_{12}:M_{24};\bZ)_{(2)} = \bZ/8 \oplus \bZ/4$.
\end{lemma}

\begin{proof}
The LHS spectral sequence for $\H^\bullet(C_{12}:M_{24};\bZ)_{(2)}$ begins
\begin{equation}
\label{eq:m24ss}
\begin{array}{cccccc}
 \H^0\bigl(M_{24};H^4(C_{12};\bZ)\bigr) \\
 & \H^1\bigl(M_{24};\Alt^2(C_{12}^\vee)\bigr)\\
 && \H^2\bigl(M_{24};C_{12}^\vee\bigr) \\
 0 & 0 & 0 & 0 \\ \hdashline
 * & 0 & 0 & 0 & \bZ/{4} 
\end{array}
\end{equation}
The dashed line reminds that the bottom row splits off as a direct summand, and we have used equation~\eqref{eq:coh of elem 2gp} for the values of $\H^i(C_{12};\bZ)$. The computation of $\H^4(M_{24};\bZ)_{(2)} \cong \bZ/4$ is due to~\cite{SE09}.  

$C_{12}$ is a submodule of the permutation module $\bF_2^{24}$.  In one basis of $C_{12}$, we find the ATLAS generators for $M_{24}$ act by the matrices
\begin{equation}\label{eq:C12v}
\mbox{\footnotesize$\left(\begin{array}{p{.5ex}p{.5ex}p{.5ex}p{.5ex}p{.5ex}p{.5ex}p{.5ex}p{.5ex}p{.5ex}p{.5ex}p{.5ex}p{.5ex}}
0 & 1 & 0 & 0 & 0 & 0 & 0 & 0 & 0 & 0 & 0 & 0\\
1 & 0 & 0 & 0 & 0 & 0 & 0 & 0 & 0 & 0 & 0 & 0\\
0 & 0 & 0 & 1 & 0 & 0 & 0 & 0 & 0 & 0 & 0 & 0\\
0 & 0 & 1 & 0 & 0 & 0 & 0 & 0 & 0 & 0 & 0 & 0\\
0 & 0 & 0 & 0 & 0 & 1 & 0 & 0 & 0 & 0 & 0 & 0\\
0 & 0 & 0 & 0 & 1 & 0 & 0 & 0 & 0 & 0 & 0 & 0\\
0 & 0 & 0 & 0 & 0 & 0 & 1 & 0 & 0 & 0 & 0 & 0\\
0 & 0 & 0 & 0 & 0 & 0 & 0 & 0 & 1 & 0 & 0 & 0\\
0 & 0 & 0 & 0 & 0 & 0 & 0 & 1 & 0 & 0 & 0 & 0\\
0 & 0 & 0 & 0 & 0 & 0 & 0 & 0 & 0 & 0 & 1 & 0\\
0 & 0 & 0 & 0 & 0 & 0 & 0 & 0 & 0 & 1 & 0 & 0\\
0 & 0 & 0 & 0 & 0 & 0 & 1 & 0 & 0 & 0 & 0 & 1
\end{array}\right)$}
\hspace{.5in}
\mbox{\footnotesize$\left(\begin{array}{p{.5ex}p{.5ex}p{.5ex}p{.5ex}p{.5ex}p{.5ex}p{.5ex}p{.5ex}p{.5ex}p{.5ex}p{.5ex}p{.5ex}}
0 & 0 & 1 & 0 & 0 & 0 & 0 & 0 & 0 & 0 & 0 & 0\\
0 & 1 & 1 & 0 & 0 & 0 & 0 & 0 & 0 & 0 & 0 & 0\\
1 & 0 & 1 & 0 & 0 & 0 & 0 & 0 & 0 & 0 & 0 & 0\\
0 & 0 & 0 & 0 & 1 & 0 & 0 & 0 & 0 & 0 & 0 & 0\\
0 & 0 & 0 & 0 & 0 & 0 & 1 & 0 & 0 & 0 & 0 & 0\\
0 & 0 & 0 & 0 & 0 & 0 & 0 & 1 & 0 & 0 & 0 & 0\\
0 & 0 & 0 & 1 & 0 & 0 & 0 & 0 & 0 & 0 & 0 & 0\\
0 & 0 & 0 & 0 & 0 & 0 & 0 & 0 & 0 & 1 & 0 & 0\\
1 & 0 & 1 & 1 & 0 & 1 & 1 & 1 & 1 & 0 & 0 & 0\\
0 & 0 & 0 & 0 & 0 & 1 & 0 & 0 & 0 & 0 & 0 & 0\\
1 & 0 & 0 & 1 & 0 & 1 & 1 & 1 & 0 & 0 & 1 & 0\\
0 & 0 & 0 & 0 & 0 & 0 & 0 & 0 & 0 & 0 & 0 & 1
\end{array}\right)$}
\end{equation}
The action of these matrices by right multiplication on row vectors gives $C_{12}^\vee$.  This is a suitable input for Cohomolo, which verifies
\[
\H^2(M_{24};C_{12}^\vee) \cong \bZ/2.
\]
The action of \eqref{eq:C12v} on $\Alt^2(C_{12}^\vee)$ and $\Alt^3(C_{12}^\vee)$ gives larger matrix representations which are still small enough to be handled by Cohomolo:
\[
\H^1\bigl(M_{24};\Alt^2(C_{12}^\vee)\bigr) \cong \H^0\bigl(M_{24};\Alt^3(C_{12}^\vee)\bigr) \cong \bZ/2.
\]
(The nontrivial fixed point in $\Alt^3(C_{12}^\vee)$ is the \define{triple intersection} sending Golay codewords $a,b,c \subset \{1,\dots,24\}$ to $|a \cap b \cap c| \mod 2$.)

Since $\H^0$ is left exact and $\H^0(M_{24};C_{12}^\vee) = \H^0(M_{24};\Alt^2(C_{12}^\vee)) = 0$, Lemma \ref{lem:filtration} shows the map $\H^0(M_{24};\H^4(C_{12};\bZ)) \to \H^0(M_{24};\Alt^3(C_{12}^\vee))$ is an injection.
By Lemma~\ref{Co0 order 2 lower bound}, there is an element of order $8$ in $\H^4(2^{12}:M_{24};\bZ)$. From this we can conclude first that $\H^0(M_{24};H^4(C_{12};\bZ))$ is non-zero, hence isomorphic to $\bZ/2$, and second that all of the displayed groups in \eqref{eq:m24ss} survive to $E_{\infty}$ and participate in a nontrivial extension, proving the Lemma.
\end{proof}

To complete the proof of Theorem~\ref{conway theorem}, it suffices to prove:

\begin{lemma} \label{p2}
 The restriction map $\H^4(\Co_0;\bZ) \to \H^4(2^{12}:M_{24};\bZ)$ is not a surjection.
\end{lemma}

\begin{proof}
 Let $\Perm : M_{24} \to \rO(24)$ denote the $\bR$-linear permutation representation of $M_{24}$ --- it is isomorphic to $\Leech \otimes \bR\vert_{M_{24}}$.  As for $\Co_0$, $M_{24}$ is superperfect and every real representation has a unique lift to $\Spin$ --- we will abuse notation and denote this lift by $\Perm:M_{24} \to \Spin(24)$ as well.  Of course we have
\[
\frac{p_1}2(\Perm) = \frac{p_1}2(\Leech|_{M_{24}}) \qquad \text{in } \H^4(M_{24};\bZ) 
\]
but we claim that the pullback of $\frac{p_1}{2}(\Perm)$ along the projection $(C_{12}:M_{24}) \to M_{24}$ does not extend to a class on $\Co_0$.
 
 To see this, let $g \in M_{24}$ be an element of the $M_{24}$-conjugacy class 2B (acting with cycle structure $2^{12}$).  Then $\langle g\rangle \to \Spin(24)$ projects to 12 copies of the trivial representation plus 12 copies of the sign representation, and because of this
 \[
 \frac{p_1}2(\Perm)|_{\langle g\rangle} = \frac{12}{4} \neq 0 \qquad \text{in } \H^4(\langle g \rangle;\bZ) \cong \bZ/2.
 \] 
Of course, for $x \in C_{12}$, $\Perm|_{\langle x\rangle}$ is the trivial representation, and so $\frac{p_1}2(\Perm)|_{\langle x\rangle} = 0$.

If $\frac{p_1}2(\Perm)$ is the restriction of a class $\lambda \in \H^4(\Co_0;\bZ)$, then its restriction $\lambda|_{\langle x\rangle}$ depends only on the conjugacy class of $x$ in $\Co_0$. Thus to show that such a $\lambda$ does not exist, it suffices to find an element $x \in 2^{12}$ conjugate in $\Co_0$ to $g \in M_{24}$.
 
 According to \cite{MR723071}, $\Co_0$ has four conjugacy classes of order $2$: the central element $c\in \Co_0$; two that cover the conjugacy class 2A in $\Co_1$; and one that covers the conjugacy class 2C in $\Co_1$. The conjugacy class 2B in $\Co_1$ lifts to order $4$ in $\Co_0$. These classes are distinguished by their traces:
 \[\trace(c,\Leech) = -24, \quad \trace(\text{lifts of 2A},\Leech) = \pm 8, \quad \trace(\text{lift of 2C},\Leech) = 0. \]
The codewords in $C_{12}$ with Hamming length $8$, $12$, $16$, and $24$ respectively act on $\Leech$ with trace $8$, $0$, $-8$, and $-24$.

Thus codewords of Hamming length $12$ are conjugate in $\Co_0$ to elements of $M_{24}$ of $M_{24}$-conjugacy class 2B. But $\frac{p_1}2(\Perm) \in \H^4(2^{12}:M_{24};\bZ)$ vanishes on all codewords and does not vanish on class 2B, and so cannot extend to a cohomology class on $\Co_0$.
\end{proof}

\section{\texorpdfstring{$\H^4(M_{24};\bZ)$}{H4(M24;Z)}, \texorpdfstring{$\H^4(M_{23};\bZ)$}{H4(M23;Z)}, and \texorpdfstring{$\H^4(\Co_1;\bZ)$}{H4(Co1;Z)}}

\label{section: other groups}

In outline, our proof of Theorem~\ref{conway theorem} had the following structure, for $G = \Co_0$:
\begin{enumerate}
  \item Quickly determine $\H^4(G;\bZ)_{(p)} = 0$ for large primes $p$ for which $G$ has a very simple $p$-Sylow subgroup.
  \item Find a characteristic class $\alpha \in \H^4(G;\bZ)$ and a small subgroup $C \subset G$ such that $\alpha|_C$ has large order. This provides a lower bound on $\H^4(G;\bZ)$.
  \item For small primes $p$, find a subgroup of $G$ containing the $p$-Sylow of shape $p^n:J$. Compute the $E_2$-page of the LHS spectral sequence for $\H^4(p^n:J;\bZ)$. This provides a preliminary upper bound on $\H^4(G;\bZ)_{(p)}$.
  \item Find a characteristic class in $\H^4(J;\bZ)$ whose pullback to $\H^4(p^n:J;\bZ)$ distinguishes elements that are conjugate in $G$, and so doesn't extend to $G$. This narrows the upper bound on $\H^4(G;\bZ)_{(p)}$ to agree with the lower bound, completing the proof.
\end{enumerate}

In this section we will discuss, via examples, the extent to which this strategy works for other groups.  We will give new proofs of the isomorphisms $\H^4(M_{24};\bZ) = \bZ/12$ and $\H^4(M_{23};\bZ) = 0$ essentially following the steps (1)--(4).  But we will see that the strategy fails for $\Co_1$ --- it turns out that the bound from step (3) is insufficiently sharp. A more serious version of this obstacle is encountered when trying to compute $\H^4$ of the Monster, see~\cite[\S 3.5]{JFmoonshine} for some discussion.  Nevertheless for $\Co_1$, we are able to deduce $\H^4(\Co_1;\bZ) = \bZ/12$ from a simple reduction to Theorem~\ref{conway theorem}.
\medskip

We first confirm \eqref{eq:Ellis-M24}, due originally to \cite{SE09}:

\begin{theorem}\label{m24 theorem}
  The group cohomology $\H^4(M_{24};\bZ)$ is isomorphic to $\bZ/12$. Let $\Perm : M_{24} \to S_{24}$ denote the defining permutation representation and $\Perm \otimes \bC$ the corresponding complex representation.  
  The Chern class $c_2(\Perm \otimes \bC)$ generates a subgroup of index $2$ in $\H^4(M_{24};\bZ)$; since $\H^2(M_{24};\bZ) = \H^3(M_{24};\bZ) = 0$, the real representation $\Perm \otimes \bR$ carries a unique spin structure, and $\frac{p_1}2(\Perm \otimes \bR)$ is a distinguished generator of $\H^4(M_{24};\bZ)$. Let $\langle 12\rB\rangle \subset M_{24}$ denote the cyclic subgroup generated by an element of conjugacy class $12\rB$; then the restriction map $\H^4(M_{24};\bZ) \to \H^4(\langle 12\rB \rangle;\bZ)$ is an isomoprhism.
\end{theorem}
\begin{proof}
  For $p=7$ and $23$, the $p$-Sylow in $M_{24}$ is contained in a maximal subgroup isomorphic to $L_2(p)$, giving $\H^4(M_{24};\bZ)_{(p)} = 0$ as in Lemma~\ref{lemma p7}. The $3$-, $5$-, and $11$-Sylows in $M_{24}$ are contained in a  subgroup isomorphic to $M_{12}$. A by-hand computation (in, for example, \cite{MilgramTezuka}) gives $\H^4(M_{12};\bZ)_{(5)} = \H^4(M_{12};\bZ)_{(11)} = 0 $ and $\H^4(M_{12};\bZ)_{(3)} = \bZ/3$. 
  
  For reasons that will become apparent, during the proof we will denote the degree-24 permutation representation of $M_{24}$ as $\Perm_{24}$.
  Conjugacy class $12\rB$ acts with cyclic structure $12^2$, from which one computes that $c_2(\Perm_{24} \otimes \bC)|_{\langle 12\rB\rangle}$ has order $6$. Since the permutation representation is spin, $c_2(\Perm_{24}\otimes \bR)$ is even. This completes the proof of the Theorem for odd primes and provides the claimed upper bound for $p=2$.
  
  The $2$-Sylow in $M_{24}$ is contained in a maximal subgroup of shape $2^4:A_8$; the action of $A_8$ on $2^4$ uses the exceptional isomorphism $A_8 \cong \mathrm{GL}(4,\bF_2)$. Using Cohomolo but not HAP, one  confirms:
  $$ 
  \H^2(A_8;\H^2(2^4;\bZ)) \cong \H^1(A_8;\H^3(2^4;\bZ)) \cong \bZ/2 \quad \text{and} \quad
  \H^0(A_8;\H^4(2^4;\bZ)) = 0.$$
  Furthermore, $\H^4(A_8;\bZ) \cong \bZ/12$.
  These provide the $E_2$-page of the LHS spectral sequence, from which we learn that $\H^4(2^4:A_8)_{(2)} \cong \bZ/4 \oplus X$, where the first summand is $\H^4(A_8;\bZ)_{(2)}$ and where $X$ is one of the groups $\bZ/2$, $(\bZ/2)^2$, or $\bZ/4$.
  
  Let $\Perm_8 \otimes \bC$ denote the $8$-dimensional complex permutation representation of $A_8$. Then $c_2(\Perm_8 \otimes \bC)$ generates $\H^4(A_8;\bZ)$~\cite{MR878978}. We claim that the pullback $2c_2(\Perm_8 \otimes \bC) \in \H^4(2^4:A_8;\bZ)$ does not extend to $\H^4(M_{24};\bZ)$. 
  Indeed, let $g \in A_8$ be an element of order $4$ which has a fixed point in the degree-8 permutation representation; its cycle structure is $1^2 2^1 4^1$, and so $2c_2(\Perm_8 \otimes\bC)|_{\langle g\rangle}$ has order $2$ in $\H^4(\langle g\rangle) \cong \bZ/4$. 
  Let $h \in A_8$ have cycle structure $1^4 2^2$; then $2c_8(\Perm_8\otimes \bC) = 0$. Choose $x \in 2^4$ such that $x$ is not fixed by $h$. 
  Then $xh \in 2^4:A_8$ has order $4$ and $2c_2(\Perm_8\otimes \bC)|_{\langle xh\rangle} = 2c_2(\Perm_8\otimes \bC)|_{\langle h\rangle} = 0$. 
  But both $xh$ and $g$ are order-$4$ elements of $M_{24}$ which fix points in the degree-24 permutation representation of $M_{24}$, and there is a unique conjugacy class of such elements. Since $2c_2(\Perm_8\otimes \bC)\in \H^4(2^4:A_8;\bZ)$ distinguished conjugate-in-$M_{24}$ elements, it cannot extend to a class on $M_{24}$.
  
  We know that $\H^4(M_{24};\bZ)_{(2)}$ contains an element of order $4$, namely $\frac{p_1}2(\Perm_{24} \otimes \bR)$. If we had $X \cong \bZ/2$ or $(\bZ/2)^2$, then we would have $2\frac{p_1}2(\Perm_{24} \otimes \bR) = 2c_2(\Perm_8\otimes \bC)\in \H^4(2^4:A_8;\bZ)$, which is impossible since $2c_2(\Perm_8\otimes \bC)$ does not extend to $M_{24}$. So $X \cong \bZ/4$ and $\H^4(M_{24};\bZ)_{(2)}$ is a direct summand of $(\bZ/4)^2$ which is nonempty (since it contains $\frac{p_1}2(\Perm_{24} \otimes \bR)$) and not everything (since it does not contain $2c_2(\Perm_8\otimes \bC)$). Thus $\H^4(M_{24};\bZ)_{(2)} \cong \bZ/4$ generated by $\frac{p_1}2(\Perm_{24} \otimes \bR)$.
\end{proof}

A very similar argument applies to $M_{23}$. The computation of $\H^4(M_{23};\bZ)$ is due to Milgram~\cite{MR1736514}:

\begin{theorem}
\label{thm:milgram}
  The group cohomology $\H^4(M_{23};\bZ)$ vanishes.
\end{theorem}
\begin{proof}
  The odd Sylow subgroups are contained in subgroups of shape $3^2:8$, $5:4$, $7:3$, $11:5$ and $23:11$, where in the first case $8 \cong \bF_9^\times$ acts by multiplication on $\bF_9 \cong 3^2$; thus $\H^4(M_{24};\bZ)_{(\mathrm{odd})} = 0$ as in Lemma~\ref{lemma p7}. The $2$-Sylow is contained in a maximal subgroup isomorphic to $2^4:A_7$. Furthermore:
  $$ \H^2(A_7;\H^2(2^4;\bZ)) = \H^1(A_7;\H^3(2^4;\bZ)) = \H^0(A_7;\H^4(2^4;\bZ)) = 0.$$
  It follows that $\H^4(M_{23};\bZ)_{(2)} \to \H^4(A_7;\bZ)_{(2)} \cong \bZ/4$ is an injection onto a direct summand. But, exactly as in the proof of Theorem~\ref{m24 theorem}, $2c_2(\Perm_7\otimes \bC)$ distinguishes conjugate elements in $M_{23}$, where $\Perm_7$ denotes the defining permutation representation of $A_7$.
\end{proof}

To end this section, let us show $\H^4(\Co_1;\bZ) = \bZ/12$ by a different argument.  

\begin{theorem}\label{conway corollary}
  $\H^4(\Co_1;\bZ) \cong \bZ/{12}$. 
\end{theorem}

It can be shown that the $276$-dimensional representation of $\Co_1$ is Spin.
(Indeed, this is the adjoint rep of $\operatorname{PSO}(24) \supset \Co_1$, and the adjoint rep of $\operatorname{PSO}(2n)$ is Spin when $n = 0$ or $1$ mod $4$.)
It follows from the table at the end of Section~\ref{sec:second-chern-classes} that an explicit generator for $\H^4(\Co_1)$ is $\frac{p_1}2(276)$.

\begin{proof}
Consider the LHS spectral sequence for the extension $\Co_0 = 2.\Co_1$. Its $E_2$ page begins:
$$ 
\begin{array}{cccccc}
 2 \\
 0 & 0 & 0\\
 2 & 0 & \H^2(\Co_1;2) \\ 
 0 & 0 & 0 & 0 \\
 \bZ & 0 & 0 & \H^3(\Co_1;\bZ) & \H^4(\Co_1;\bZ) 
\end{array} 
=
\begin{array}{cccccc}
 2 \\
 0 & 0 & 0\\
 2 & 0 & 2 \\ 
 0 & 0 & 0 & 0 \\
 \bZ & 0 & 0 & 2 & \H^4(\Co_1;\bZ) 
\end{array} 
$$
Specifically, write $E_r^{i,j}$ for the $(i,j)$th entry on the $E_r$ page. On the $i+j=3$ diagonal, the groups $E_2^{21}$, $E_2^{12}$, and $E_2^{03}$ vanish. It follows that $E_{\infty}^{04} = E_2^{04} = \H^4(\Co_1;\bZ)$. On the $i+j = 4$ diagonal, the groups $E_2^{22}$ and $E_2^{04}$ are copies of $\bZ/2$, while $E_2^{31}$ and $E_2^{13}$, hence also $E_\infty^{31}$ and $E_\infty^{13}$, vanish. 

We claim that $E_\infty^{04}$ also vanishes. Indeed, the center $\bZ/2 \subset \Co_0$ acts on $\Leech \otimes \bR$ as 24 copies of the sign representation, so $\frac{p_1}{2}$ vanishes there (see Theorem \ref{restrictions theorem} for a more general statement). It follows that $\H^4(\Co_0;\bZ) \to E_2^{04}$ is zero, but $E_\infty^{04} \subset E_2^{04}$ is precisely the image of this map.

In total degree $4$, the LHS filtration on $\H^4(\Co_0;\bZ) \cong \bZ/24$ reduces therefore to a short exact sequence
\[ 
0 \to \H^4(\Co_1;\bZ) \to \bZ/24 \to E_{\infty}^{22} \to 0.
\]
But $E_{\infty}^{22}$ is a subquotient of $E_{2}^{22} = \bZ/2$. We conclude that $\H^4(\Co_1;\bZ)$ is either $\bZ/12$ or $\bZ/24$.

It remains to rule out the latter option. Equivalently, we must show that the image of $\H^4(\Co_1;\bZ)$ in $\H^4(\Co_0;\bZ)$ does not contain an element of order~$8$. 
One can detect whether a class in $\H^4(\Co_0;\bZ)$ has order~$8$ by restricting to the binary dihedral group $2D_8 \subset \Co_0$. But the composition $\H^4(\Co_1;\bZ) \to \H^4(\Co_0;\bZ) \to \H^4(2D_8;\bZ)$ factors through $\H^4(D_{8};\bZ) = (\bZ/2)^2\oplus\bZ/4$. 
\end{proof}

Theorem~\ref{conway corollary} is slightly surprising if one tries to repeat 
the strategy outlined at the beginning of this Section.
 The maximal subgroup of $\Co_1$ containing the $2$-Sylow has shape $C_{11}:M_{24}$, where as above $C_{11}$ denotes the irreducible Golay code module. The $E_2$ page of the corresponding LHS spectral sequence for $\H^\bullet(C_{11}:M_{24};\bZ)_{(2)}$ can be computed as in the proof of Lemma~\ref{Co0 even lemma}:
$$ \begin{array}{cccccc}
 2 \\
 0 & 2 & 2\\
 0 & 2 & 2 \\ 
 0 & 0 & 0 & 0 \\\hdashline
 \bZ & 0 & 0 & 0 & \bZ/{4} 
\end{array} $$
The class in $\H^0(M_{24};\H^4(C_{11};\bZ))$ is the triple intersection $(a,b,c) \mapsto | a \cap b \cap c|$ --- we defined it on $C_{12} = 2.C_{11}$, but it vanishes if any element is the all-1s vector --- and the pullback $\H^0(M_{24};\H^4(C_{11};\bZ)) \to \H^0(M_{24};\H^4(C_{12};\bZ))$ is an isomorphism. For comparison, the $E_2$ page for $\H^\bullet(C_{12}:M_{24};\bZ)_{(2)}$ is
$$ \begin{array}{cccccc}
 2 \\
 0 & 2 & 2\\
 0 & 0 & 2 \\  0 & 0 & 0 & 0 \\\hdashline
 * & 0 & 0 & 0 & \bZ/{4} 
\end{array} $$
and the triple intersection extends to an element with order $8$. How, then, can $\H^4(C_{11}:M_{24};\bZ)_{(2)}$ fail to have elements of order $8$? Shouldn't the triple intersection have order $8$ there?

The answer is that the triple intersection does not survive the LHS spectral sequence for $\H^\bullet(C_{11}:M_{24};\bZ)_{(2)}$ but does for $\H^\bullet(C_{12}:M_{24};\bZ)_{(2)}$.
The extension $\H^3(C_{12};\bZ) = \Alt^2(C_{12}^\vee) = \Alt^2(C_{11}^*).C_{11}^*$ leads to a long exact sequence in $M_{24}$-cohomology:
$$\begin{tikzpicture}[anchor=base]
 \path 
 (0,1) node {$\Alt^2(C_{11}^*)$} (3,1) node {$\Alt^2(C_{12}^\vee)$} (6,1) node {$C_{11}^*$}
 (-3,0) node {$\H^0(M_{24};-)$} (-3,-1) node {$\H^1(M_{24};-)$} (-3,-2) node {$\H^2(M_{24};-)$}
 (0,0) node (a0) {$0$} (3,0) node (b0) {$0$} (6,0) node (c0) {$0$}
 (0,-1) node (a1) {$\bZ/2$} (3,-1) node (b1) {$\bZ/2$} (6,-1) node (c1) {$\bZ/2$}
 (0,-2) node (a2) {$\bZ/2$} (3,-2) node (b2) {$\bZ/2$} (6,-2) node (c2) {$\bZ/2$}
 ;
 \draw[->] (a1) -- node[auto] {$\scriptstyle \sim$} (b1);
 \draw[->] (c1) -- node[auto,swap] {$\scriptstyle \sim$} (a2);
 \draw[->] (b2) -- node[auto] {$\scriptstyle \sim$} (c2);
\end{tikzpicture}$$
In particular, the restriction map $\H^2(M_{24};\H^3(C_{11};\bZ)) \to \H^2(M_{24};\H^3(C_{12};\bZ))$ is $0$.
This provides the room needed for the $d_2$ differential $\H^0(M_{24};\H^4(C_{11};\bZ)) \to \H^2(M_{24};\H^3(C_{11};\bZ))$ to be non-zero while the $d_2$ differential $\H^0(M_{24};\H^4(C_{12};\bZ)) \to \H^2(M_{24};\H^3(C_{12};\bZ))$ is $0$.

\section{Second Chern classes of representations}
\label{sec:second-chern-classes}

Let us index the representations of $2D_8$ as in \eqref{eq:McKay}:
\begin{equation}
\label{eq:McKay}
\begin{aligned}
\begin{tikzpicture}
  \path (0,0) node (SW) {$V_1$}
        (1,1) node (A) {$V_6$}
        (0,2) node (NW) {$V_0$}
        (2,1) node (M) {$V_4$}
        (3,1) node (B) {$V_5$}
        (4,2) node (NE) {$V_2$}
        (4,0) node (SE) {$V_3$};
  \draw (A) -- (SW); \draw (A) -- (NW); \draw (A) -- (M);
  \draw (B) -- (SE); \draw (B) -- (NE); \draw (B) -- (M);
\end{tikzpicture}
\end{aligned}
\end{equation}
Thus $M = V_6$ and $M' = V_5$ in the notation of Lemma~\ref{Co0 order 2 lower bound}.  $V_0$ is the trivial representation, and $V_1,V_2,V_3$ are the nontrivial one-dimensional representations. The kernel of $V_1$ is cyclic (of order $8$), while the kernels of $V_2$ and $V_3$ are quaternion groups of order $8$.  $V_4$ is the real dihedral representation into $\mathrm{O}(2)$, the symmetries of the square.

\begin{lemma} \label{restricting to 2D8}
Let $V$ be a representation of $\Co_0$, and suppose that $V\vert_{2D_8} = \bigoplus_{i = 0}^6 n_i V_i$.  Then
\begin{equation}
\label{eq:c2restriction}
c_2(V\vert_{2D_8}) = 4n_4 + 9n_5 + n_6 \quad \text{\rm{ }mod 16}
\end{equation}
\end{lemma}

\begin{proof}
For $i = 1,2,3$, put $v_i := c_1(V_i)$.  Then $v_1,v_2,v_3$ are the three nonzero elements of $\H^2(2D_8;\bZ) \cong \bZ/2 \oplus \bZ/2$, and $v_1 = v_2 + v_3$.  As $\H^4(2D_{8};\bZ)$ is cyclic, we must either have $v_2^2 = 0$ or $v_2^2 = 8$, and similarly for $v_3$.  As $v_2$ and $v_3$ are exchanged by an outer automorphism of $2D_8$, we have $v_2^2 = v_3^2$, and therefore $v_1^2 = 0$.  

(One may also see that $v_1^2 = 0$ by observing that $V_1$ is pulled back from a one-dimensional representation of $2D_{16}$, and that the restriction map $\H^4(2D_{16};\bZ) \to \H^4(2D_8;\bZ)$, being a map from $\bZ/32$ to $\bZ/16$, must vanish on the $2$-torsion subgroup of the domain.  We are not sure whether or not $v_2^2$ and $v_3^2$ are zero, but to prove the Lemma we will not need to know.)

In the proof of Lemma \ref{Co0 order 2 lower bound}, we have already computed the total Chern classes of $V_6 = M$ and $V_5 = M'$ ---
\[
c_t(V_6) = 1 + t^2, \qquad c_t(V_5) = 1+9t^2
\]  
--- by decomposing $\Sym^3(V_6)$ as $V_6 \oplus V_5$.  To prove the Lemma, we will appeal to the following computations:
\begin{equation}
\label{eq:sublemma}
c_t(V_4) = 1+ v_1 t + 4t^2, \qquad\text{ and } \qquad v_2 v_3 = 0.
\end{equation}
The first follows from considering the decomposition of $\Sym^2(V_6)$ and the second from considering the decomposition of $\Sym^4(V_6)$:
\begin{equation}
\label{eq:sym2sym4}
\Sym^2(V_6) = V_1 \oplus V_4, \qquad \Sym^4(V_6) = V_0 \oplus V_2 \oplus V_3 \oplus V_4.
\end{equation}
Under the identification $\H^4(B\SU(2);\bZ) \cong \bZ$, we have
\[
c_2(\Sym^n(\bC^2)) = \frac{1}{6}n^3 + \frac{1}{2}n^2 + \frac{1}{3}n,\qquad \text{i.e. 1, 4, 10, 20, 35, \ldots}
\]
so \eqref{eq:sym2sym4} gives
\[
1+4t^2 = (1+v_1t)(1+v_1 t + c_2(V_4)t^2), \qquad 1 + 20t^2 = (1 + v_2t)(1+v_3t)(1+v_1 t + c_2(V_4)t^2),
\]
which gives \eqref{eq:sublemma} upon expanding and using $v_1^2 = 0$.

Now we compute \eqref{eq:c2restriction} by considering the total Chern class of the direct sum:
\[
1^{n_0}(1+v_1 t)^{n_1}(1+v_2 t)^{n_2}(1+v_3 t)^{n_3}(1+v_1 t + 4t^2)^{n_4}(1+9t^2)^{n_5}(1+t^2)^{n_6}
\]
In fact $n_2 = n_3$ for every representation of $\Co_0$ --- this can be seen from the merging of conjugacy classes of $2D_8$ in $\Co_0$, or just by checking each irreducible representation one by one.  As 
\[
(1+v_2 t)(1+v_3 t) = 1 + (v_2+ v_3)t + v_2 v_3 t^2 = 1+v_1 t,
\]
the total Chern class of $V\vert_{2D_8}$ is
\begin{equation}
\label{eq:expand-this}
(1+v_1 t)^{n_1+n_2}(1+v_1 t + 4t^2)^{n_4}(1+9t^2)^{n_5}(1+t^2)^{n_6}
\end{equation}
Since $v_1^2 = 0$, the coefficient of $t^2$ in the expansion of \eqref{eq:expand-this} is $4n_4 + 9n_5 + n_6$.
\end{proof}

Let $\mathfrak{c}_1,\mathfrak{c}_2,\ldots,\mathfrak{c}_{167}$ be GAPs ordering of the conjugacy classes of $\Co_0$, in its library of character tables.  Then $\mathfrak{c}_1$ is the identity element and $\mathfrak{c}_2$ is the central element, and:
\begin{enumerate}
\item $\mathfrak{c}_5$ is the unique conjugacy class that squares to the central element,
\item $\mathfrak{c}_{21}$ is the unique conjugacy class that squares to $\mathfrak{c}_5$,
\item $\mathfrak{c}_{13}$ is the unique conjugacy class of order $3$ whose trace on $\Leech$ is zero.
\end{enumerate}
If $V$ is any complex representation of $\Co_0$, we have $c_2(V) = k(V) c_2(\Leech \otimes \bC)$ for some $k(V) \in \bZ/12$.
Theorem~\ref{conway theorem} implies that $k$ depends only on 
\begin{equation}
\label{eq:five-classes}
\trace(\mathfrak{c}_i, V)\text{ for }i \in \{1,2,5,21,13\}.
\end{equation}
The numbers $k(V_1),\ldots,k(V_{167})$, where $V_1,\ldots,V_{167}$ are the irreducible characters in the order that they appear in GAP's library,
are recorded in the table 
on the next page, along with the traces at~\eqref{eq:five-classes}.
Note that $V \mapsto k(V)$ 
 is a group homomorphism $R(\Co_0) \to \bZ/12$, since $c_1(V) = 0$ for every complex representation $V$ of $\Co_0$.

Incidentally, 153 of the 167 irreducible representations of $\Co_0$ are real --- it is easier to list the fourteen exceptions, which have GAP indices
\[
\begin{array}{cccccccccccccc}
17, &
18, &
27, &
28, &
121, &
122, &
125, &
126, &
128, &
129, &
135, &
136, &
142, &
143.
\end{array}
\]
Any real representation of $\Co_0$, irreducible or not, has a unique lift from $\Co_0 \to \mathrm{O}(n)$ to $\Co_0 \to \Spin(n)$, and therefore has a fractional Pontryagin class of the form $k'(V) \cdot \frac{p_1}{2}(\Leech \otimes \bR)$, with $k'(V) = -k(V) \text{ mod }12$.  The discussion of \S\ref{p12 discussion} shows that $k'$ is a homomorphism $\mathit{RO}(\Co_0) \cong \bZ^{160} \to \bZ/24$.  One could compute it if one knew the ``supercohomology'' of $2D_8$, and the string obstruction map on $\mathit{RO}(2D_8)$, but we have not done the computation.

\newpage
\[
\tiny
\begin{array}{c|c}
\begin{array}{ccccccc}
i & \mathfrak{c}_1 &  \mathfrak{c}_2 & \mathfrak{c}_5 & \mathfrak{c}_{21} & \mathfrak{c}_{13} & k(V_i)\\
1 & 1 & 1 & 1 & 1 & 1 & 0 \\
2 & 276 & 276 & 12 & 0 & 0 & 10 \\
3 & 299 & 299 & -13 & -1 & -1 & 2 \\
4 & 1771 & 1771 & 51 & 7 & 7 & 6 \\
5 & 8855 & 8855 & 15 & -1 & -7 & 10 \\
6 & 17250 & 17250 & 90 & -6 & 0 & 10 \\
7 & 27300 & 27300 & 156 & 0 & 0 & 2 \\
8 & 37674 & 37674 & -78 & -6 & 0 & 6 \\
9 & 44275 & 44275 & -77 & 7 & 1 & 8 \\
10 & 80730 & 80730 & -78 & 6 & 0 & 6 \\
11 & 94875 & 94875 & 235 & -1 & 21 & 8 \\
12 & 313950 & 313950 & -26 & 14 & 21 & 6 \\
13 & 345345 & 345345 & 377 & 1 & -21 & 6 \\
14 & 376740 & 376740 & -364 & 0 & 36 & 0 \\
15 & 483000 & 483000 & -600 & 0 & 0 & 4 \\
16 & 644644 & 644644 & 1092 & 28 & 49 & 4 \\
17 & 673750 & 673750 & 350 & -14 & -35 & 6 \\
18 & 673750 & 673750 & 350 & -14 & -35 & 6 \\
19 & 822250 & 822250 & -910 & -14 & 43 & 2 \\
20 & 871884 & 871884 & 1612 & 28 & 27 & 0 \\
21 & 1434510 & 1434510 & 246 & 6 & 0 & 6 \\
22 & 1450449 & 1450449 & 1001 & -7 & 63 & 6 \\
23 & 1771000 & 1771000 & 40 & 0 & 28 & 4 \\
24 & 1821600 & 1821600 & 352 & 0 & -36 & 0 \\
25 & 2055625 & 2055625 & 1625 & 13 & 85 & 8 \\
26 & 2417415 & 2417415 & -1001 & -21 & 84 & 8 \\
27 & 2464749 & 2464749 & -987 & 21 & 0 & 6 \\
28 & 2464749 & 2464749 & -987 & 21 & 0 & 6 \\
29 & 2816856 & 2816856 & -1080 & 0 & 0 & 0 \\
30 & 2877875 & 2877875 & 715 & -1 & 35 & 6 \\
31 & 4100096 & 4100096 & 0 & 0 & -64 & 8 \\
32 & 5494125 & 5494125 & 1365 & 21 & 0 & 2 \\
33 & 5801796 & 5801796 & -364 & -28 & 63 & 0 \\
34 & 7628985 & 7628985 & 1001 & -7 & -27 & 0 \\
35 & 9221850 & 9221850 & -1598 & -14 & 27 & 6 \\
36 & 9669660 & 9669660 & 364 & -28 & 21 & 8 \\
37 & 12432420 & 12432420 & 2028 & 0 & 0 & 6 \\
38 & 16347825 & 16347825 & -351 & -27 & 0 & 0 \\
39 & 20083140 & 20083140 & -924 & 0 & 0 & 0 \\
40 & 21049875 & 21049875 & -2325 & 15 & 0 & 6 \\
41 & 21528000 & 21528000 & 4160 & 0 & 36 & 0 \\
42 & 21579129 & 21579129 & -2223 & -27 & 162 & 6 \\
43 & 23244375 & 23244375 & 2975 & 35 & 105 & 10 \\
44 & 24174150 & 24174150 & 5278 & 14 & 21 & 2 \\
45 & 24667500 & 24667500 & -1300 & -28 & 105 & 4 \\
46 & 24794000 & 24794000 & 560 & 0 & -70 & 4 \\
47 & 25900875 & 25900875 & 3835 & -41 & -63 & 0 \\
48 & 31574400 & 31574400 & 3328 & 0 & 84 & 8 \\
49 & 40166280 & 40166280 & 2520 & 0 & 0 & 0 \\
50 & 40370176 & 40370176 & 8192 & 0 & 112 & 4 \\
51 & 44013375 & 44013375 & -1665 & 27 & 0 & 0 \\
52 & 46621575 & 46621575 & -6201 & 27 & 0 & 0 \\
53 & 51571520 & 51571520 & 0 & 0 & 56 & 4 \\
54 & 55255200 & 55255200 & -4576 & 0 & -84 & 0 \\
55 & 57544344 & 57544344 & -936 & 0 & 0 & 0 \\
56 & 60435375 & 60435375 & -3185 & 7 & -105 & 8 \\
57 & 65270205 & 65270205 & -3003 & 21 & 0 & 6 \\
58 & 66602250 & 66602250 & 4290 & 6 & 0 & 6 \\
59 & 77702625 & 77702625 & -975 & 21 & 0 & 0 \\
60 & 83720000 & 83720000 & 0 & -56 & 20 & 0 \\
61 & 85250880 & 85250880 & 7488 & 0 & 0 & 0 \\
62 & 91547820 & 91547820 & 4212 & 0 & 0 & 6 \\
63 & 100725625 & 100725625 & 6825 & -7 & -140 & 4 \\
64 & 106142400 & 106142400 & 11648 & 56 & 216 & 0 \\
65 & 109882500 & 109882500 & -9100 & -28 & 105 & 0 \\
66 & 150732800 & 150732800 & 4096 & 64 & 104 & 8 \\
67 & 163478250 & 163478250 & -1950 & -6 & 0 & 6 \\
68 & 184184000 & 184184000 & 4160 & 0 & -28 & 0 \\
69 & 185912496 & 185912496 & 11088 & 0 & 162 & 0 \\
70 & 185955000 & 185955000 & 4200 & 0 & 0 & 4 \\
71 & 191102976 & 191102976 & -4096 & -64 & 216 & 0 \\
72 & 201451250 & 201451250 & -4550 & 14 & 35 & 2 \\
73 & 205395750 & 205395750 & -5810 & -14 & 189 & 6 \\
74 & 207491625 & 207491625 & 8865 & -27 & -162 & 6 \\
75 & 210974400 & 210974400 & 0 & 0 & 0 & 0 \\
76 & 215547904 & 215547904 & 0 & 0 & 64 & 0 \\
77 & 219648000 & 219648000 & 0 & 0 & 0 & 8 \\
78 & 241741500 & 241741500 & 5460 & 0 & 0 & 10 \\
79 & 247235625 & 247235625 & 6825 & 21 & 0 & 0 \\
80 & 251756505 & 251756505 & -6903 & -27 & 0 & 0 \\
81 & 257857600 & 257857600 & -11648 & 56 & -56 & 4 \\
82 & 259008750 & 259008750 & 9750 & -6 & 0 & 6 \\
83 & 267014475 & 267014475 & 8827 & 7 & 189 & 0 \\
84 & 280280000 & 280280000 & 0 & -56 & 140 & 4 \\
\end{array}
&
\begin{array}{ccccccc}
i & \mathfrak{c}_1 &  \mathfrak{c}_2 & \mathfrak{c}_5 & \mathfrak{c}_{21} & \mathfrak{c}_{13} & k(V_i)\\
85 & 292953024 & 292953024 & -7488 & 0 & 0 & 0 \\
86 & 299710125 & 299710125 & 6565 & 13 & -216 & 6 \\
87 & 302176875 & 302176875 & 2275 & 7 & -105 & 6 \\
88 & 309429120 & 309429120 & -11648 & 0 & 84 & 0 \\
89 & 326956500 & 326956500 & -2340 & 0 & 0 & 6 \\
90 & 360062976 & 360062976 & 0 & 0 & 0 & 8 \\
91 & 387317700 & 387317700 & -4356 & 0 & 0 & 6 \\
92 & 402902500 & 402902500 & -9100 & 28 & -35 & 8 \\
93 & 464257024 & 464257024 & 4096 & -64 & -56 & 0 \\
94 & 469945476 & 469945476 & 4004 & -28 & -189 & 0 \\
95 & 469945476 & 469945476 & 9828 & 0 & 0 & 0 \\
96 & 483483000 & 483483000 & -3640 & 0 & -84 & 4 \\
97 & 502078500 & 502078500 & -1260 & 0 & 0 & 0 \\
98 & 503513010 & 503513010 & 3354 & 6 & 0 & 6 \\
99 & 504627200 & 504627200 & -4096 & 64 & 56 & 4 \\
100 & 522161640 & 522161640 & 2184 & 0 & 0 & 0 \\
101 & 551675124 & 551675124 & -9828 & 0 & 0 & 6 \\
102 & 24 & -24 & 0 & 0 & 0 & 1 \\
103 & 2024 & -2024 & 0 & 0 & 8 & 3 \\
104 & 2576 & -2576 & 0 & 0 & 8 & 2 \\
105 & 4576 & -4576 & 0 & 0 & -8 & 8 \\
106 & 40480 & -40480 & 0 & 0 & -8 & 4 \\
107 & 95680 & -95680 & 0 & 0 & -8 & 0 \\
108 & 170016 & -170016 & 0 & 0 & 0 & 4 \\
109 & 299000 & -299000 & 0 & 0 & 8 & 5 \\
110 & 315744 & -315744 & 0 & 0 & 0 & 4 \\
111 & 351624 & -351624 & 0 & 0 & 0 & 11 \\
112 & 388080 & -388080 & 0 & 0 & 0 & 6 \\
113 & 789360 & -789360 & 0 & 0 & 48 & 2 \\
114 & 1841840 & -1841840 & 0 & 0 & 8 & 6 \\
115 & 1937520 & -1937520 & 0 & 0 & 0 & 6 \\
116 & 4004000 & -4004000 & 0 & 0 & 56 & 4 \\
117 & 5051904 & -5051904 & 0 & 0 & -48 & 0 \\
118 & 6446440 & -6446440 & 0 & 0 & 112 & 7 \\
119 & 6446440 & -6446440 & 0 & 0 & -56 & 11 \\
120 & 7104240 & -7104240 & 0 & 0 & 0 & 6 \\
121 & 9152000 & -9152000 & 0 & 0 & -112 & 8 \\
122 & 9152000 & -9152000 & 0 & 0 & -112 & 8 \\
123 & 11051040 & -11051040 & 0 & 0 & 48 & 0 \\
124 & 13156000 & -13156000 & 0 & 0 & -56 & 0 \\
125 & 15002624 & -15002624 & 0 & 0 & -112 & 4 \\
126 & 15002624 & -15002624 & 0 & 0 & -112 & 4 \\
127 & 15471456 & -15471456 & 0 & 0 & 0 & 4 \\
128 & 16170000 & -16170000 & 0 & 0 & 0 & 10 \\
129 & 16170000 & -16170000 & 0 & 0 & 0 & 10 \\
130 & 17050176 & -17050176 & 0 & 0 & 0 & 0 \\
131 & 17310720 & -17310720 & 0 & 0 & 0 & 8 \\
132 & 18987696 & -18987696 & 0 & 0 & 0 & 6 \\
133 & 19734000 & -19734000 & 0 & 0 & 48 & 2 \\
134 & 34155000 & -34155000 & 0 & 0 & 0 & 9 \\
135 & 40370176 & -40370176 & 0 & 0 & 112 & 4 \\
136 & 40370176 & -40370176 & 0 & 0 & 112 & 4 \\
137 & 44204160 & -44204160 & 0 & 0 & 168 & 4 \\
138 & 49335000 & -49335000 & 0 & 0 & 0 & 1 \\
139 & 50519040 & -50519040 & 0 & 0 & 120 & 8 \\
140 & 51571520 & -51571520 & 0 & 0 & -112 & 8 \\
141 & 59153976 & -59153976 & 0 & 0 & 0 & 9 \\
142 & 59153976 & -59153976 & 0 & 0 & 0 & 9 \\
143 & 59153976 & -59153976 & 0 & 0 & 0 & 9 \\
144 & 62790000 & -62790000 & 0 & 0 & 120 & 2 \\
145 & 67358720 & -67358720 & 0 & 0 & -112 & 4 \\
146 & 106260000 & -106260000 & 0 & 0 & -120 & 0 \\
147 & 112519680 & -112519680 & 0 & 0 & -168 & 8 \\
148 & 139243104 & -139243104 & 0 & 0 & 0 & 0 \\
149 & 161161000 & -161161000 & 0 & 0 & 112 & 11 \\
150 & 190417920 & -190417920 & 0 & 0 & 168 & 0 \\
151 & 210496000 & -210496000 & 0 & 0 & 112 & 0 \\
152 & 215547904 & -215547904 & 0 & 0 & 64 & 0 \\
153 & 230230000 & -230230000 & 0 & 0 & 160 & 2 \\
154 & 247543296 & -247543296 & 0 & 0 & 168 & 0 \\
155 & 282906624 & -282906624 & 0 & 0 & -48 & 4 \\
156 & 287006720 & -287006720 & 0 & 0 & -112 & 0 \\
157 & 313524224 & -313524224 & 0 & 0 & -160 & 4 \\
158 & 342752256 & -342752256 & 0 & 0 & 0 & 0 \\
159 & 351624000 & -351624000 & 0 & 0 & 168 & 4 \\
160 & 394680000 & -394680000 & 0 & 0 & -120 & 4 \\
161 & 464143680 & -464143680 & 0 & 0 & 0 & 0 \\
162 & 485760000 & -485760000 & 0 & 0 & 120 & 0 \\
163 & 517899096 & -517899096 & 0 & 0 & 0 & 9 \\
164 & 557865000 & -557865000 & 0 & 0 & 0 & 3 \\
165 & 655360000 & -655360000 & 0 & 0 & 160 & 0 \\
166 & 805805000 & -805805000 & 0 & 0 & -280 & 3 \\
167 & 1021620600 & -1021620600 & 0 & 0 & 0 & 9 \\
\quad
\end{array}
\end{array}
\]

\newpage

\section{Restrictions to cyclic subgroups}\label{section restrictions}
\label{sec:six}

In this section we will give a formula for the restriction maps 
\begin{equation}
\label{eq:cyc-res}
\H^4(\Co_0;\bZ) \to \H^4(C;\bZ),
\end{equation}
where $C \subset \Co_0$ is any cyclic subgroup.  The domain is cyclic of order $24$, and has a distinguished generator $\frac{p_1}{2}$.  The codomain is also cyclic, of the same order as $C$.  It does not always have a distinguished generator but we give a naming scheme for the elements of its $24$-torsion subgroup that does not require any choices --- for each $k \in \bZ$ with $24k \in |C|\bZ$, there is well-defined class $kt^2 \in \H^4(C;\bZ)$.  Here $t^2$ is the cup-square of any generator $t \in \H^2(C;\bZ)$.  The fact that $kt^2$ is independent of $t$ (when $24k = 0$ modulo the order of $C$) is a consequence of what Conway and Norton call the ``defining property of $24$'' \cite[\S 3]{MR554399}: that $a^2 = 1$ mod 24 whenever $a$ is invertible mod $24$.  

Thus we may report \eqref{eq:cyc-res} by reporting an integer $k \in \bZ$ such that $\frac{p_1}{2}$ is carried to $kt^2$.  Theorem~\ref{restrictions theorem} gives a formula for $k$ in terms of the characteristic polynomial of (any generator of) $C$, regarded as a $24 \times 24$ matrix.  
That  a general formula should exist follows from the discussion in \S\ref{p12 discussion}, but our formula will apply only to the image of $\Co_0 \hookrightarrow \rO(24)$.

Actually we give the formula in terms of Frame's encoding~\cite{MR0269751} of the characteristic polynomial.
Since each element $g \in \Co_0$ preserves a lattice, its characteristic polynomial $\det(g - \lambda)$ factors uniquely as $\prod_{d|o(g)} (1-\lambda^d)^{r_d}$ for some integers $r_d \in \bZ$, and the \define{Frame shape} of $g$ is the formal expression $\prod_{d|o(g)} d^{r_d}$. Frame shapes generalize cycle structures of permutations. The Frame shapes of all elements in $\Co_0$ were computed in \cite[p.\ 355]{MR780666}: the 167 conjugacy classes in $\Co_0$ merge to only 160 different Frame shapes.

Let $\ell(g)$ denote the smallest $d$ such that the exponent $r_d$ of $d$ in the Frame shape of $g$ is non-zero. For example, $\ell(g) = 1$ if and only if $\trace(g,\bR^{24}) \neq 0$. If $g$ is a permutation matrix, then $\ell(g)$ is the length of the smallest cycle in $g$. Let $\epsilon(g) = \pm1$ record the sign of the exponent~$r_{\ell(g)}$.  We say that $g$ is \define{balanced} if there exists an $N$ such that $r_d = \epsilon(g) r_{N/d}$  for all $d$.  The notion of a balanced Frame shape specializes to the notion of a balanced cycle type in the sense of \cite[p.\ 1 item (B)]{MR554399}; in particular every element of $M_{24} \subset \Co_0$ is balanced.  The conjugacy class $8\rB$ in $\Co_0$ has Frame shape $2^{-4}8^4$, and so $\ell(8\rB) = 2$, $\epsilon(8\rB) = -1$, and $8\rB$ is balanced. Conjugacy class $4\rD$ has Frame shape $2^{-4} 4^8$ and is not balanced. The following result summarizes our calculations of $\frac{p_1}2(\Leech \otimes \bR)|_{\langle g\rangle}$:

\begin{theorem}\label{restrictions theorem}
  Suppose that $g \in \Co_0$, and use notation $t, o(g),\dots$ as above.
  \begin{enumerate}
    \item If $\ell(g) = 1$, then $\frac{p_1}2(\Leech \otimes \bR)|_{\langle g\rangle} = 0$. \label{restrictions theorem traceful case}
    \item If $g$ is balanced, then 
    $$ \frac{p_1}2(\Leech \otimes \bR)|_{\langle g\rangle} = \frac{\epsilon(g) o(g) }{ \ell(g)} t^2 $$ \label{restrictions theorem balanced case}
    \item If $g$ is not balanced, then $\frac{p_1}2(\Leech \otimes \bR)|_{\langle g\rangle} = 0$. \label{restrictions theorem unbalanced case} \qed
  \end{enumerate}
\end{theorem}
Statement~(\ref{restrictions theorem traceful case}) is a consequence of~(\ref{restrictions theorem balanced case}) and~(\ref{restrictions theorem unbalanced case}), as $o(g) t^2 = 0$. We don't know any a priori reason for Theorem~\ref{restrictions theorem} to hold: all three statements (\ref{restrictions theorem traceful case}--\ref{restrictions theorem unbalanced case}) fail in general for other lattice-preserving elements of $\Spin(24)$. Our proof is case-by-case: we computed $\frac{p_1}2(\Leech \otimes \bR)$ for all 160 Frame shapes associated with the Conway group.

Specifically, we found a factorization of each $\langle g \rangle \subset \Co_0 \hookrightarrow \rO(24)$ through $\SU(12) \to \Spin(24)$.  Suppose more generally that $V : \bZ/n \to \rO(2m)$ is given. Then the $2m$ eigenvalues of $V(g)$ lie on $\mathrm{U}(1) \subset \bC$ and come in $m$ complex-conjugate pairs.  To factorize $V$ through $\rU(m)$ is equivalent to selecting one eigenvalue from each of these pairs.  To factorize through $\SU(m)$ one must select them such that their product is $1$. We found that for $\bZ/n \subset \Co_0$, this is always possible, although the ``obvious'' factorization through $\rU(12)$ sometimes fails. For example, for element $4\rH \in \Co_0$, with Frame shape $4^6$, the ``obvious'' factorization through $\rU(12)$ uses a matrix with determinant~$-1$; any ``correct'' factorization through $\SU(12)$ has spectrum which is not invariant under complex conjugation.

Having factored $\Leech \otimes \bR|_{\langle g\rangle} = V : \bZ/o(g) \to \rO(2m) = \rO(24)$ through $\bZ/o(g) \overset W \to \SU(m) \mono \rO(2m)$, we may compute $\frac{p_1}2(V)$ quickly. Indeed, $\SU(m)$ is simply connected, and so injects into $\Spin(2m)$, and the restriction map $\H^4(B\Spin(2m);\bZ) \to \H^4(B\SU(m);\bZ)$ carries $\frac{p_1}{2}$ to $-c_2$. The Cartan formula gives a recipe for the Chern classes of $W$ in terms of the eigenvalues of $W(g)$. That is how we proved Theorem~\ref{restrictions theorem}.

\section{Restrictions to umbral subgroups} \label{umbral section}

Every even unimodular lattice $L \subset \bR^{24}$ is isometric to either $\Leech$ or to one of the twenty-three Niemeier lattices.  If $L$ is a Niemeier lattice, it is characterized up to isometry by its root system $\Phi_L \subset L$ --- the vectors of length $2$ in $L$ --- and the real span of $\Phi_L$ is all of $\bR^{24}$.  Reflection through the root vectors generates a Weyl group $W_L$, which is normal in the full isometry group $\Aut(L)$.  Let $U_L := \Aut(L)/W_L$ denote the quotient group.  We will follow \cite{MR3449012} and call $U_L$ an ``umbral group''; it is called the ``glue group'' $G_1.G_2$ in \cite{MR1662447}.  For instance, the Mathieu group $M_{24}$ is an umbral group (with $L$ of type $A_1^{24}$), as is the Schur cover $2M_{12}$ of the Mathieu group $M_{12}$ ($L$ of type~$A_2^{12}$). 

For each Niemeier lattice $L$ there is a preferred (``holy'') conjugacy class of embeddings $U_L \mono \Co_0$. Two of them, for $L$ of type $A_1^{24}$ and $A_2^{12}$, have already been mentioned in \S\ref{preliminaries on Conway group} and used in Sections~\ref{prime 3} and \ref{prime 2}.  In general the theory of root systems shows that a choice of simple roots $\Delta_L \subset \Phi_L$ induces a splitting $\Aut_L = W_L:U_L$, where $U_L$ acts by faithfully permuting $\Delta_L$ and preserving its graph structure (the Dynkin diagram) --- see \cite[\S VI.1--VI.4]{MR1890629}.  Once $\Delta_L \subset L$ is fixed, the corresponding holy construction \cite[Ch.\ 24]{MR1662447} \cite[\S7]{MR791880}
 outputs a distinguished $U_L$-stable lattice $L_0 \subset \bR^{24}$, with $L_0 \cap L$ of finite index in $L$, that is isometric to Leech.  Since $\Co_0$ has no outer automorphisms the composite $U_L \mono \Aut(L_0) \cong \Co_0$ is well-defined up to conjugacy.  In this section we make some comments about the restriction map
 \begin{equation}
 \label{eq:coum}
 \H^4(\Co_0;\bZ) \to \H^4(U_L;\bZ).
 \end{equation}

The coefficients of various famous $q$-series are integer linear combinations of entries from the character tables of umbral groups, a phenomenon called \define{umbral moonshine} in \cite{MR3271175,MR3449012}. The umbral moonshine problem is to find a family of quantum field theories $\VL$, on which the umbral groups act, that would explain (by taking characters) this phenomenon. 
 These $U_L$-actions would induce cohomology classes $\alpha_L \in \H^3(U_L;\mathrm{U}(1)) \cong \H^4(U_L;\bZ)$, which we will call \define{anomalies} based on~\cite{Wen2013}.
  These anomalies have largely been characterized, in \cite{GPRV,CLW}, even in advance of knowing what $\VL$ is: in all cases the restriction of $\alpha_L$ to a cyclic subgroup $\langle g \rangle \subset U_L$ can be extracted from the modularity properties (the multiplier system) of the $q$-series corresponding to $g$ --- see \cite[\S3.3]{GPRV} and \cite[\S6]{MR3539377} --- and for all but three of the umbral groups, $\H^4(U_L;\bZ)$ is detected on cyclic subgroups.  (The exceptions are $A_2^{12}, A_3^8$, and $A_6^4$). 
\medskip

In this section we check that for a number of $L$, $\alpha_L$ is in the image of \eqref{eq:coum}, and in fact
\begin{equation}
\label{eq:not-always}
\alpha_L = \epsilon(L) \frac{p_1}{2}(\Leech \otimes \bR) \vert_{U_L}
\end{equation}
for a scalar $\epsilon(L)$ that generates $\bZ/24$ --- we warn that we are not sure that it is true in general, and do not propose any particular relationship between the $V^L$ and the Conway group, but we do hope that some of our calculations will be useful for moonshine.  For example for $L = A_1^{24}$ or 
\begin{equation}
\label{eq:cddo}
L  \in \{A^4_6, A^2_{12}, D^4_6,D^3_8,D^2_{12}, D_{24}\},
\end{equation}
we find $\epsilon(L) = -1$.  For the list \eqref{eq:cddo}, Cheng--Duncan and Duncan--O'Desky have had some qualified success in realizing $V^L$ as a free theory --- at least, for solving what Duncan calls the ``meromorphic module problem.''  One consequence of our calculations is that there is a cohomological obstruction to solving the meromorphic module problem with a free theory, and that this obstruction is not trivial for $A_3^8$.

\begin{theorem} \label{mathieu moonshine anomaly}
  Under the standard isomorphism $\H^4(G;\bZ) \cong \H^3(G;\rU(1))$ for $|G|<\infty$ given by the Bockstein for the map $x \mapsto \exp(2\pi i x)$, the restriction of $\frac{p_1}2(\Leech \otimes \bR) \in \H^4(\Co_0;\bZ)$ to $M_{24}$ is minus the anomaly  $\alpha \in \H^3(M_{24};\rU(1))$ computed by \cite{GPRV}.
\end{theorem}

\begin{proof}
Given a finite group $G$ and $g\in G$ of order $o(g)$, consider the $3$-cycle
$$ \gamma_g = \sum_{i = 0}^{o(g)-1} g \otimes g^i \otimes g
$$
in the bar complex for $G$.  (If we consider the $G$-bundle on a 3d lens space $\SU(2)/o(g)$, whose monodromy around the nontrivial loop is $g$, then the homology class of $\gamma_g$ is the image under the classifying map $\SU(2)/o(g) \to BG$ of the fundamental class.)   
The anomaly $\alpha \in \H^3(M_{24};\rU(1))$ of~\cite{GPRV} is characterized by the property that for every $g\in M_{24}$, the pairing $\H^3(M_{24};\rU(1)) \otimes \H_3(M_{24}) \to \rU(1)$ takes $\alpha \otimes \gamma_g$ to $\exp(-2\pi i / \ell(g))$. In \cite{GPRV}, $\ell(g)$ is defined as the length of the shortest cycle in the degree-$24$ permutation representation --- the notation is consistent with the $\ell(g)$ in \S\ref{sec:six}, since the cycle type of the permutation and the Frame shape of its permutation matrix coincide.

Let $\tau \in \H^1(\langle g\rangle;\bR/\bZ) = \hom(\langle g\rangle;\bR/\bZ)$ denote the homomorphism sending $g \mapsto 1/o(g) + \bZ$, let $\beta : \H^k(G;\bR/\bZ) \to \H^{k+1}(G;\bZ)$ denote the Bockstein, and let $t = \beta(\tau) \in \H^2(\langle g\rangle;\bZ)$. Then $t$ can be represented by the cocycle 
$$ t \left(g^i\otimes g^j\right) = \begin{cases} 0, & i+j<o(g) \\ 1, & i+j \geq o(g) \end{cases}, \qquad i,j \in \{0,\dots,n-1\}.$$

Under the Bockstein identification $\H^4(\langle g\rangle;\bZ) \cong \H^3(\langle g \rangle;\bR/\bZ)$, the cocycle $t^2$ is carried to $\tau \cup t$, where $\cup : \H^1(\langle g\rangle;\bR/\bZ) \otimes \H^2(\langle g\rangle;\bZ) \to \H^3(\langle g \rangle;\bR/\bZ)$ denotes the cup product.
We calculate:
\[
\sum_{i=0}^{o(g)-1} \left[ \tau \cup t\right] \left(g \otimes g^i \otimes g\right) = \sum_{i=0}^{o(g)-1} \tau(g) \cdot t \left( g^i \otimes g \right) = \frac1{o(g)} \cdot 1
\]
since only the $i=o(g)-1$ term provides a non-zero value to $t\left( g^i \otimes g \right)$.
Since every $g\in M_{24}$ is balanced with $\epsilon(g) = +1$, the Theorem follows from part~(\ref{restrictions theorem balanced case}) of Theorem~\ref{restrictions theorem}.
\end{proof}

The papers \cite{UmbralDtype,UmbralAtype} construct super vertex algebras that explain some but not all of the umbral moonshine phenomena for the Niemeier lattices $L$ of type $A_3^8$, $A_4^6$, $A_6^4$, $A_{12}^2$, $D_6^4$, $D_8^3$, $D_{12}^2$, and $D_{24}$. 
These super vertex algebras are all of the following type. 
Let $U'_L = U_L$ when $L$ is of type $A_6^4$, $A_{12}^2$, $D_6^4$, $D_8^3$, $D_{12}^2$, and $D_{24}$, and let $U'_L$ denote the unique (up to conjugacy) maximal subgroup of $U_L$ isomorphic to $\mathrm{SL}_2(7)$ when $L = A_3^8$ or to $S_3 \times 4$ when $L = A_4^6$.
Two finite-dimensional complex representations, called in those papers $\mathfrak{b}^+$ and $\mathfrak{a}^+$, of $U'_L$ are selected. Specifically, they take:
  $$\begin{array}{rrrr} 
  L & U'_L & \mathfrak{b}^{+} & \mathfrak{a}^{+} \\ \hline \\[-10pt]
  A_3^8 & \mathrm{SL}_2(7) & \bC^4 & \bC^3 \\[2pt]
  A_4^6 & S_3 \times (\bZ/4) & (\bR^2 \boxtimes -i) \oplus (\mathrm{sign} \boxtimes i) & \mathrm{triv} \boxtimes (-1 \oplus i)  \\[2pt]
  A_6^4   &  2A_4  &  \bC^2 & \bC^1 \\[2pt]
  A_{12}^2  &  \bZ/4    &  i \oplus -i    &  -1 \\[2pt]
  D_6^4  &   S_4   &    \bR^3 & \bR^2 \\[2pt]
  D_8^3   &  S_3  &    \bR^2 &          \mathrm{sign} \\[2pt]
  D_{12}^2  &  S_2   &    \mathrm{sign} \oplus \mathrm{sign} &    \mathrm{triv} \\[2pt]
  D_{24}   &   \mathrm{triv} &     \mathrm{triv} &             \mathrm{triv} 
  \end{array}$$
A free field theory, also called a ``$\beta \gamma b c$ system'' (see for example \cite[Chapters 11 and 12]{MR2082709}), is then built from these representations: it consists of free bosons valued in $\mathfrak{b}^+ \oplus \mathfrak{b}^-$ and free fermions valued in $\mathfrak{a}^+ \oplus \mathfrak{a}^-$, where $\mathfrak{b}^- = (\mathfrak{b}^+)^*$ and $\mathfrak{a}^- = (\mathfrak{a}^+)^*$.  A physical argument shows that the anomaly of the $G$-action on such a system is $c_2(\mathfrak{b}^+) - c_2(\mathfrak{a}^+)$, provided that $c_1(\mathfrak{b}^+) = c_1(\mathfrak{a}^+)$.  (The field theories of \cite{UmbralDtype,UmbralAtype} also have some auxiliary free fermions, valued in a vector space called $\mathfrak{e}$, on which $G$ acts trivially. These do not affect the anomaly.) 

  We briefly explain the names for representations in the table.  
  By ``$\mathrm{sign}$'' and ``$\mathrm{triv}$'' we mean the sign and trivial representations of symmetric groups.  
  The $-1$ and $\pm i$ in the $A_4^6$ and $A_{12}^2$ rows denote the one-dimensional representations of $\bZ/4$ in which the generator acts with that eigenvalue. In the $\mathfrak{b}^+$ column, the representation $\bR^{n-1}$ of $S_n$ is the nontrivial submodule of the permutation representation.  In the $\mathfrak{a}^+$ column, the representation $\bR^2$ is the pullback of this representation of $S_3$ along the surjective homomorphism $S_4 \to S_3$ (the ``resolvent cubic'' of Galois theory).  The $\bC^n$s in the $A_3^8$ and $A_4^6$ rows are irreducible complex $n$-dimensional representations, that are specified up to simultaneous complex conjugation as follows.
  For $L = A_3^8$, these are chosen so that, if an order-7 element of $U'_L$ acts on $\mathfrak{a}^+$ with eigenvalue $\lambda$, then it acts on $\mathfrak{b}^+$ with trace $-\bar\lambda$.
  For $L = A_6^4$, these are chosen so that, if an order-3 element of $U_L$ acts on $\mathfrak{a}^+$ with eigenvalue $\lambda$, then it acts on $\mathfrak{b}^+$ with trace $-\bar\lambda$.

\begin{theorem} \label{thm umbral consistency check}
  For $L$ of type $A_6^4$, $A_{12}^2$, $D_6^4$, $D_8^3$, $D_{12}^2$, and $D_{24}$, for the $U_L$-representations $\mathfrak{b}^+$ and $\mathfrak{a}^+$ in \cite{UmbralDtype,UmbralAtype}, we have
  $c_1(\mathfrak{b}^+) = c_1(\mathfrak{a}^+)$ and
   $-\frac{p_1}2(\Leech \otimes \bR)|_{U_L} = c_2(\mathfrak{b}^+) - c_2(\mathfrak{a}^+)$.
\end{theorem}

\begin{proof}
For $L$ of type $A_{12}^2$, $D_6^4$, $D_8^3$, $D_{12}^2$, and $D_{24}$, classes in $\H^4(U_L;\bZ)$ are determined by their restrictions to cyclic subgroups.  For all umbral groups, $\Leech \otimes \bR|_{U_L}$ is a permutation representation of $U_L$ on the nodes of the Dynkin diagram for the root system of $L$, and the Frame shape of an element is the cycle type of this permutation.  For a given $U_L$ one can therefore check the Theorem by proving that $c_2(\mathfrak{a}^+) - c_2(\mathfrak{b}^+)$ restricts to $\frac{o(g)}{\ell(g)} t^2$ for every $g \in U_L$, by Theorem \ref{restrictions theorem}.

We will describe the case $L = D_6^4$, where $U_L = S_4$.  The other cases from $\{A_{12}^2,D_6^4,D_8^3,D_{12}^2,D_{24}\}$ can be handled similarly.  The following table lists the nontrivial conjugacy classes of $S_4$ in terms of their cycle structures on the defining degree-$4$ permutation, their Frame shapes as elements of $\Co_0$, and their eigenvalues in $\mathfrak{b}^+$ and $\mathfrak{a}^+$. For each $g\in S_4$, it then lists the values of 
  $\frac{p_1}2(\Leech \otimes \bR)|_{\langle g\rangle}$, $c_2(\mathfrak{b}^+)|_{\langle g \rangle}$, and $c_2(\mathfrak{a}^+)|_{\langle g\rangle}$ as multiples of the canonical generator $t^2$ of $\H^4(\langle g\rangle;\bZ)$; $\frac{p_1}2$ is computed using Theorem~\ref{restrictions theorem}, and the $c_2$s are immediate:
  $$ \begin{array}{c|ccc|ccc}
   S_4 & \Leech & \mathfrak{b}^+ & \mathfrak{a}^+ & \frac{p_1}2(\Leech) & c_2(\mathfrak{b}^+) & c_2(\mathfrak{a}^+) \\[2pt] 
   \hline &&&&&& \\[-10pt]
   1^2 2^1 
   & 1^8 2^8 & -1 \oplus 1 \oplus 1 & -1 \oplus 1
   & 0 & 0 & 0
   \\
   2^2 
   & 2^{12} & -1 \oplus -1 \oplus 1 & 1 \oplus 1
   & 1 & 1 & 0
   \\
   1^1 3^1 
   & 1^6 3^6 & \lambda \oplus \bar\lambda \oplus 1 & \lambda \oplus \bar\lambda
   & 0 & -1 & -1
   \\
   4^1 
   & 4^6 & i \oplus -1 \oplus -i & -1 \oplus 1
   & 1 & -1 & 0
  \end{array} $$
  In the above table, $\lambda$ denotes a cube root of unity. For all $g\in S_4$, we find that $-\frac{p_1}2(\Leech \otimes \bR) = c_2(\mathfrak{b}^+) - c_2(\mathfrak{a}^+) \mod \operatorname{order}(g)$.
  
The only remaining case is $L = A_6^4$, which we study for the remainder of the proof. This is perhaps the most interesting case, since it is one of the three Niemeier lattices for which classes in $\H^4(U_L;\bZ)$ are not determined by their restrictions to cyclic subgroups \cite{CLW}. 

  We first verify that $c_1(\mathfrak{b}^+) = c_1(\mathfrak{a}^+)$. The characters of the two representations are:
  $$ \begin{array}{r|rrrrrrr}
  & 1\rA & 2\rA & 4\rA & 3\rA & 6\rA & 3\rB & 6\rB \\
   \hline \\[-10pt]
  \mathfrak{b}^+ & 2 & -2 & 0 & -\bar\lambda & \lambda & -\lambda & \bar\lambda \\
  \mathfrak{a}^+ & 1 & 1 & 1 & \lambda & \bar\lambda & \bar\lambda & \lambda 
  \end{array}$$
  with $\lambda = \exp(2\pi i /3)$. In particular, class $4\rA$ acts on $\mathfrak{b}^+$ with eigenvalues $i \oplus -i$, hence with determinant $1$; class $3\rA$ acts on $\mathfrak{a}^+$ with eigenvalue $\lambda$ and on $\mathfrak{b}^+$ with eigenvalues $1 \oplus \lambda$.  Also, note that $c_2(\mathfrak{a}^+) = 0$, as $\mathfrak{a}^+$ is one-dimensional --- we are left with proving
\begin{equation}
\label{eq:vercingetorix}
c_2(\mathfrak{b}^+) = -\frac{p_1}{2}(\Leech \otimes \bR\vert_{U_L})
\end{equation}

The group $U_L \cong 2A_4$ is the McKay correspondent of $E_6$, and it has a unique faithful representation into $\SU(2)$, let us denote it by $V$.  Meanwhile $\mathfrak{b}^+$ is one of the other two $2$-dimensional representations of $2A_4$.  As for any finite subgroup of $\SU(2)$, $\H^4(U_L;\bZ)$ is generated by $c_2(V)$ with order $|U_L| = 24$.

We analyze $\Leech \otimes \bR|_{U_L}$ by thinking of it as a permutation representation on the nodes of the Dynkin diagram. This action has three orbits: the nodes at the edges of the $A_6$-components, the nodes at distance one from the edges, and the nodes at distance two from the edges. These orbits are abstractly isomorphic as $U_L$-sets; we will refer to this degree-$8$ permutation representation as $\Perm_8$, so that $\Leech \otimes \bR|_{U_L} \cong \Perm_8^{\oplus 3} \otimes \bR$. Since $\Perm_8 \otimes \bR$ is a Spin representation, we have:
$$ \frac{p_1}{2}(\Leech \otimes \bR\vert_{U_L}) = 3 \frac{p_1}2(\Perm_8 \otimes \bR). $$

The restriction maps to the cyclic subgroup $\bZ/3 \subset U_L$ of order $3$ and to the quaternion subgroup $Q_8 \subset U_L$ of order $8$ give an isomorphism $\H^4(U_L;\bZ) \to \H^4(\bZ/3;\bZ) \oplus \H^4(Q_8;\bZ) \cong (\bZ/3) \oplus (\bZ/8)$.  We will prove \eqref{eq:vercingetorix} by proving 
\begin{equation}
\label{eq:viriathus}
c_2(\mathfrak{b}^+\vert_{\bZ/3}) = -3\frac{p_1}{2}(\Perm_8 \otimes \bR\vert_{\bZ/3}) = 0 \text{ and } c_2(\mathfrak{b}^+\vert_{Q_8}) = -3\frac{p_1}{2}(\Perm_8 \otimes \bR\vert_{Q_8}).
\end{equation}
For the left equation in \eqref{eq:viriathus}, note that $\mathfrak{b}^+\vert_{\bZ/3}$ splits as the sum of two one-dimensional representations, one of which is trivial, and so $c_2(\mathfrak{b}^+|_{\bZ/3}) = 0$.

We turn to the right equation in \eqref{eq:viriathus}. We have an isomorphism $\mathfrak{b}^+\vert_{Q_8} \cong V\vert_{Q_8}$, since they are both irreducible two-dimensional representations. Let $W$ denote the underlying four-dimensional real representation of $V$ and let $X$, $Y$, and $Z \cong X \otimes Y$ denote the three non-trivial one-dimensional real representations of $Q_8$. $\Perm_8\vert_{Q_8}$ is isomorphic to the regular representation $\bZ[Q_8]$, which over $\bR$ decomposes as $W \oplus X \oplus Y \oplus Z \oplus 1$.  The real representations $W$ and $X \oplus Y \oplus Z$ are each Spin, so $\frac{p_1}2(W \oplus X \oplus Y \oplus Z) = \frac{p_1}2(W) + \frac{p_1}2(X \oplus Y \oplus Z)$.

To compute $\frac{p_1}2(W)$ we can use the observation that the action $W : Q_8 \to \SO(4)$ factors through $V : Q_8 \to \SU(2)$, and so $$\frac{p_1}2(W) = -c_2(V).$$
  
To compute $\frac{p_1}2(X \oplus Y \oplus Z)$, note that $X \oplus Y \oplus Z : Q_8 \to \SO(3)$ is nothing but the image of $V : Q_8 \to \SU(2) \cong \Spin(3)$ under the canonical map $\Spin(3) \to \SO(3)$. The group $\Spin(3)$ is unusual among Spin groups in the following way. The stable class $\frac{p_1}2 \in \H^4(B\Spin(3);\bZ)$, pulled back from the generator of $\H^4(B\Spin(\infty);\bZ)$, does not generate $\H^4(B\Spin(3);\bZ) \cong \bZ$, whereas $\frac{p_1}2$ generates $\H^4(B\Spin(n);\bZ)$ for $n\geq 5$. Rather, $\frac{p_1}2 \in \H^4(B\Spin(3);\bZ)$ is twice the generator. To see this, note that $\H^4(B\Spin(4);\bZ) \cong \bZ^2$ and that the restriction maps $\H^4(B\Spin(5);\bZ) \to \H^4(B\Spin(4);\bZ) \to \H^4(B\Spin(3);\bZ)$ are the diagonal embedding followed by addition. On the other hand, think of $\Spin(3)$ as $\SU(2)$: $\H^4(B\SU(2);\bZ)$ is generated by $c_2$. After checking signs, one finds that $\frac{p_1}2 = -2 c_2$ as classes in $\H^4(B\Spin(3);\bZ) = \H^4(B\SU(2);\bZ)$. Restricting to $Q_8$ gives $$\frac{p_1}2(X \oplus Y \oplus Z) = -2 c_2(V).$$

All together, we have
$$ \frac{p_1}2(W) + \frac{p_1}2(X \oplus Y \oplus Z) = -3 c_2(V) \in \H^4(Q_8; \bZ).$$
Multiplying both sides by $-3$ gives \eqref{eq:viriathus} as desired.
\end{proof}

We conclude with some calculations that show (consistent with calculations in \cite{CHVZ}) that the anomaly $\alpha_L$ does not agree with $-\frac{p_1}{2}$ for $L$ of type $A_3^8$ or $A_4^6$.

\medskip

For $L = A_4^6$, we calculate the restrictions of $\frac{p_1}2(\Leech \otimes \bR)$ and $c_2(\mathfrak{b}^+) - c_2(\mathfrak{a}^+)$ to the elements $3\rA$ and $4\rA$ in $U'_L = S_3 \times 4$:
  $$ \begin{array}{cc|ccc|ccc}
  \text{Name} & S_3 \times 4 & \Leech & \mathfrak{b}^+ & \mathfrak{a}^+ & \frac{p_1}2(\Leech) & c_2(\mathfrak{b}^+) & c_2(\mathfrak{a}^+) \\[2pt] 
   \hline &&&&&& \\[-10pt]
  3\rA & (3^1, 0) & 3^8 & \lambda \oplus \bar\lambda \oplus 1 & 1 \oplus 1 & 1 & -1 & 0
  \\
  4\rA & (1^3, 1) & 4^6 & -i \oplus -i \oplus i & -1 \oplus i & 1 & -1 & 2
  \end{array} $$
Here $\lambda = \exp(2\pi i /3)$. We have recorded the eigenvalues of each element in $\mathfrak{b}^+$ and $\mathfrak{a}^+$ and their Frame shapes in the permutation representation on the nodes of the $A_4^6$ Dynkin diagram (equivalently the Leech representation). We also indicate each element as a pair in $S_3 \times \bZ/4$, where the class in $S_3$ is indicated by its Frame shape and we write $\bZ/4$ additively. The calculation of $\frac{p_1}2(\Leech)$ is from Theorem~\ref{restrictions theorem} and the calculations of the $c_2$s are routine; we record the values as multiples of the square generator of $\H^4(\bZ/3;\bZ)$ or $\H^4(\bZ/4;\bZ)$. In particular, whereas when restricting to element $3\rA$, we do find $-\frac{p_1}2(\Leech \otimes \bR) = c_2(\mathfrak{b}^+) - c_2(\mathfrak{a}^+)$ as in Theorem~\ref{thm umbral consistency check}, when restricting to $4\rA$ we find the opposite sign.
It follows that the restrictions  to element $12\rA = (3^1,1)$ of $\frac{p_1}2(\Leech \otimes \bR)$ and $c_2(\mathfrak{b}^+) - c_2(\mathfrak{a}^+)$ differ by a factor of $\epsilon(L) = 5$.

\medskip

The final case studied in \cite{UmbralAtype} and not covered by Theorem~\ref{thm umbral consistency check} is $L = A_3^8$.  The calculations in~\cite{CHVZ} suggest that  \eqref{eq:not-always} should hold with $\epsilon(L) = 1$, at least after restricting to cyclic subgroups.  Note, however, that the Niemeier lattice $L$ of type $A_3^8$ is one of the three Niemeier lattices for which classes in $\H^4(U_L;\bZ)$ are not determined by their restrictions to cyclic subgroups \cite{CLW}; this remains true for the maximal subgroup $U'_L \cong \mathrm{SL}(2,7)$ studied in \cite{UmbralAtype}. 

It is easy to check that $c_1(\mathfrak{a}^+) = c_1(\mathfrak{b}^+)$, so to test \eqref{eq:not-always}, we must calculate 
\begin{equation}
\label{eq:must-calc}
\frac{p_1}2(\Leech\otimes \bR)|_{U'_L} \text{ and } c_2(\mathfrak{b}^+) - c_2(\mathfrak{a}^+).
\end{equation}
Since $\H^4(\mathrm{SL}(2,7);\bZ) \cong \bZ/48$ has only $2$- and $3$-primary torsion, we may compare these classes by comparing their restrictions to the $2$- and $3$-Sylow subgroups of $\SL(2,7)$.  The $3$-Sylow is cyclic of order 3, its generator acts on $\Leech \otimes \bR$ with Frame shape $1^6 3^6$, and on $\mathfrak{b}^+$ and $\mathfrak{a}^+$ with Frame shapes $3^1$ and $1^1 3^1$ respectively; thus the $3$-primary parts of \eqref{eq:must-calc} both vanish.  It remains to compare the $2$-primary parts.

The $2$-Sylow in $\SL(2,7)$ is isomorphic to $2D_8$, let us index its complex representations just as in \S\ref{sec:second-chern-classes}.  As with all umbral groups, the action of $U'_L \cong \mathrm{SL}(2,7)$ on $\Leech\otimes \bR$ is a permutation representation on the nodes of the Dynkin diagram; after restricting to the $2$-Sylow $2D_8 \subset \mathrm{SL}(2,7)$, it is the sum of the regular representation of $2D_8$ with the regular representation of $D_8$. This $24$-dimensional real representation is the underlying real representation of a complex representation $V$ that splits over $2D_8$ as
\[
V = V_0 \oplus (V_1 \oplus V_2 \oplus V_3) \oplus 2V_4 \oplus V_5 \oplus V_6
\]
and so $\frac{p_1}2(\Leech \otimes \bR|_{2D_8}) = -c_2(V).$  The representations $\mathfrak{b}^+$ and $\mathfrak{a}^+$ decompose as
$$ \mathfrak{b}^+|_{2D_8} = V_5 \oplus V_6 
, \quad\quad
\mathfrak{a}^+|_{2D_8} = V_1 \oplus V_2 \oplus V_3.$$
As observed in the proof of Lemma~\ref{restricting to 2D8}, $c_2(V_1 \oplus V_2 \oplus V_3)$ has order $2$ in $\H^4(2D_8;\bZ)$. Thus to compare $\frac{p_1}2(\Leech \otimes \bR|_{2D_8})$ with $c_2(\mathfrak{b}^+) - c_2(\mathfrak{a}^+)$, it suffices to compare $\pm c_2(2V_4 \oplus V_5 \oplus V_6)$ with $c_2(V_5 \oplus V_6)$. According to Lemma~\ref{restricting to 2D8},
\[
c_2(2V_4 \oplus V_5 \oplus V_6) = 2c_2(V_6), \quad\quad c_2(V_5 \oplus V_6) = 10c_2(V_6).
\]
which certainly differ, even up to sign. Rather, we find that, for $L = A_3^8$,
$$ \alpha_L|_{U'_L} = 5 \frac{p_1}2(\Leech \otimes \bR)|_{U'_L},$$
suggesting that $\epsilon(A_3^8) = 5$. (The above equation holds whether $c_2(V_1 \oplus V_2 \oplus V_3) = 0$ or $c_2(V_1 \oplus V_2 \oplus V_3) = 8 c_2(V_6)$.)

But restrictions to cyclic groups can only determine a class in $\H^4(2D_8;\bZ) \cong \bZ/16$ modulo $8$, and so we confirm the calculation of \cite{CHVZ} that, for $L = A_3^8$, the multipliers in umbral moonshine agree with those that would be given if the anomaly were $+\frac{p_1}2(\Leech \otimes\bR)|_{U_L}$, i.e.\ if we had $\epsilon(L) = 1$.

To conclude, we remark that these cohomological methods do explain why in the case $L = A_3^8$, the authors of \cite{UmbralAtype} were unable to find a ``free field'' realization of the entire umbral group $U_L \cong 2^4:\mathrm{GL}(3,2)$ reproducing the umbral moonshine functions. Indeed, all Chern classes in the 2-primary part of $\H^4(U_L;\bZ)$ have order four or less, but the previous calculations show that the anomaly for the $A_3^8$ moonshine of \cite{UmbralAtype} has order $8$.


\newcommand{\etalchar}[1]{$^{#1}$}

\end{document}